\documentclass[11pt]{article}

\usepackage{ifpdf}
\ifpdf 
    \usepackage[pdftex]{graphicx}   % to include graphics
    \usepackage[pdftex,     % sets up hyperref to use pdftex driver
            plainpages=false,   % allows page i and 1 to exist in the same document
            breaklinks=true,    % link texts can be broken at the end of line
            colorlinks=true,
            pdftitle=My Document
            pdfauthor=My Good Self
           ]{hyperref} 
    \usepackage{thumbpdf}
    \usepackage{mathtools,xparse}
      \usepackage{amsthm,amsmath,amssymb}
\usepackage{mathrsfs}
\usepackage{bm}
    \usepackage{amsmath}
	\usepackage{amsthm}    
    \usepackage{subfigure}
      \usepackage{float}
    \newtheorem{theorem}{Theorem}[section]
\newtheorem{lemma}{Lemma}[section]
    \newtheorem{remark}{Remark}[section]
    \usepackage{mathrsfs}
     \usepackage{thumbpdf}
    \usepackage{indentfirst}
\usepackage{amsmath}
\usepackage{titlesec}
\usepackage{setspace}
\usepackage{geometry}
\usepackage{scalerel,stackengine}

\DeclareMathOperator{\sign}{sgn}
\stackMath
\newcommand\widecheck[1]{%
\savestack{\tmpbox}{\stretchto{%
  \scaleto{%
    \scalerel*[\widthof{\ensuremath{#1}}]{\kern-.6pt\bigwedge\kern-.6pt}%
    {\rule[-\textheight/2]{1ex}{\textheight}}%WIDTH-LIMITED BIG WEDGE
  }{\textheight}% 
}{0.5ex}}%
\stackon[1pt]{#1}{\scalebox{-1}{\tmpbox}}%
}
\newcommand{\RNum}[1]{\uppercase\expandafter{\romannumeral #1\relax}}
\DeclarePairedDelimiter{\norm}{\lVert}{\rVert}
\NewDocumentCommand{\normL}{ s O{} m }{%
  \IfBooleanTF{#1}{\norm*{#3}}{\norm[#2]{#3}}_{L_2(\Omega)}%
}
\else 
    \usepackage{graphicx}       % to include graphics
    \usepackage{hyperref}       % to simplify the use of \href
\fi 
\newcommand{\vect}[1]{\bold{#1}}

\usepackage{enumitem}

\title{ Discontinuous Galerkin methods for stochastic Maxwell equations with multiplicative noise}
\author{	Jiawei Sun\footnote{Department of Mathematics, The Ohio State University,
		Columbus, OH 43210, USA. E-mail: sun.2261@buckeyemail.osu.edu.} \and
		Chi-Wang Shu\footnote{Division of Applied Mathematics, Brown University, Providence, RI 02912, USA. 
E-Mail: chi-wang\_shu@brown.edu. The work of this author is partially supported by NSF grant DMS-2010107
and AFOSR grant FA9550-20-1-0055.} 
		\and Yulong Xing\footnote{Department of Mathematics, The Ohio State University,
		Columbus, OH 43210, USA. E-mail: xing.205@osu.edu. The work of this author is partially supported by the NSF grant DMS-1753581.}}
\date{}

\begin{document}
\maketitle
\begin{abstract}
In this paper we propose and analyze finite element discontinuous Galerkin methods for the one- and two-dimensional stochastic Maxwell equations with multiplicative noise. 
The discrete energy law of the semi-discrete DG methods were studied. Optimal error estimate of the semi-discrete method is obtained for the one-dimensional case, and the two-dimensional case on both rectangular meshes and triangular meshes under certain mesh assumptions. Strong Taylor 2.0 scheme is used as the temporal discretization. Both one- and two-dimensional numerical results are presented to validate the theoretical analysis results.
\end{abstract}


\section{Introduction}
Stochastic Maxwell equations are commonly used to model microscopic origins of randomness in electromagnetism. The concept of stochastic Maxwell equations was firstly introduced by Rytov et al. in \cite{Rytov} to describe the fluctuations of an electromagnetic field. In \cite{ORD}, Ord et al. studied the Mark Kac random walk model and modified it into the Max field equations in 1+1 dimensions. Such model describes most of the telegraph equations, and the author's modification of the Kac model constructed a strong connection between the telegraph and Maxwell equations.  
In \cite{HSY 2010} Horsin et al. applied an abstract approach and a constructive approach by generalizing the Hilbert uniqueness method to analyze the approximate controllability of the stochastic Maxwell equations. Furthermore in \cite{BL2010} the deterministic and stochastic integrodifferential equations in Hilbert spaces were studied and the well-posedness for the Cauchy problem of the integrodifferential equation were analyzed. Such results were motivated from mathematical modeling of electromagnetics fields in complex random media.

%\YX{These are all numerical study based on uncertainty quantification. Any theoretical results? See for example the review article in \cite{LZ2019}}
Numerical methods are often used to solve Maxwell equations with various forms of stochasticity.
In \cite{BS2015}, Benner et al. studied the time-harmonic Maxwell's equations with some 
uncertainty in material parameters. They compared stochastic collocation and Monte Carlo simulation, and computed a reduced model in order to 
lower the computational cost. In \cite{J2009} Jung considered the wave and Maxwell equations with fluctuations by random change in media parameters. The author used polynomial chaos Galerkin projections to develop the evolution of probability distribution function. Later in \cite{JKMN2014} Jung et al. studied two-dimensional transverse magnetic Maxwell equations with multiple random interfaces. They applied the polynomial chaos projection method, and computed the stochastic and deterministic part separately. Furthermore, stochastic collocation methods for metamaterial Maxwell's equations are studied by Li et al. in \cite{LFL2018}. They considered the equations with random coefficient and random initial conditions. They also developed regularity analysis for stochastic metamaterial Maxwell's equations.

In the recent years, there have been many studies on various numerical methods for the stochastic Maxwell equations with either additive or multiplicative noises. 
In \cite{HJ2017}, Hong et al. studied the the stochastic Maxwell equations driven by a multiplicative noise, which are given as follows:
\begin{eqnarray}\label{multiHJmul}
  \epsilon d\vect{E}=\nabla\times\vect{H}-\lambda \vect{H}\circ dW,~ 
  \mu d\vect{H}=-\nabla\times\vect{E}+\lambda\vect{E}\circ dW, 
\end{eqnarray}
where $\bf{H}$ represents the magnetic field, $\bf{E}$ stands for the electric field, $\lambda$ is the scale of the noise, $\circ$ denotes the stochastic integral in Stratonovich sense, and $W$ is a space time mixed Wiener process.  
They studied the multi-symplectic structure and the energy conservation 
law for \eqref{multiHJmul}, and developed a fully discrete numerical method which 
conserves both multi-symplecticity and energy in the discrete level. 
Furthermore, Cohen et al. in \cite{CCHS} analyzed the general form of stochastic 
Maxwell equations with multiplicative noise, 
%\begin{eqnarray}\label{multistomax1}
%  d\mathbb{U}=\Big(M\mathbb{U}+F(\mathbb{U})\Big)dt+B(\mathbb{U})dW_t,
%\end{eqnarray}
%where $\mathbb{U}=(\vect{E}^T,\vect{H}^T)^T$, and 
%\[M=\left(\begin{array}{cc}
%0&\nabla\times\\
%-\nabla\times&0\end{array}\right),\]
%and $W_t$ is a space time mixed Wiener process.  They 
and constructed an exponential integrator 
%for \eqref{multistomax1} which reads
%\begin{eqnarray}\label{multiEx}
%  \mathbb{U}_{k+1}=S(\Delta t)\mathbb{U}_k+S(\Delta t)F(\mathbb{U}_k)\Delta t+S(\Delta t)
%  G(\mathbb{U}_k)\Delta W_k,
%\end{eqnarray}
%where $S(t)=\exp(tM)$, and showd that the scheme \eqref{multiEx} 
which has a general mean square convergence order 0.5 in time, and first order temporal rate under some assumptions.
%when $F$ and $G$ do not depend on $\mathbb{U}$, the scheme has first temporal convergence order. 
Chen et al. in \cite{CHJ2019} applied a semi-implicit scheme for the model under a general setting and showed that the proposed method has the mean-square order of 0.5 in time.
%\begin{eqnarray}\label{multistomax2}
%  d\mathbb{U}=\Big(M\mathbb{U}+F(t,\mathbb{U})\Big)dt+B(t,\mathbb{U})dW_t,
%\end{eqnarray}
%which in the form  of 
%\begin{eqnarray}\label{multisemi}
%  \mathbb{U}_{k+1}=\mathbb{U}_k+\Delta tM\mathbb{U}_{k+1}+\Delta tF(t_{k+1},\mathbb{U}_{k+1})
%  +G(t_k,\mathbb{U}_k)\Delta W_k.
%\end{eqnarray}
%The authors showed that the scheme \eqref{multisemi} has the mean-square order of 0.5 in time. 
In \cite{LZ2019}, Zhang et al. presented a nice review article to summarize numerical methods for different kinds of stochastic Maxwell equations with both additive noise and multiplicative noise, including \eqref{multiHJmul} studied in \cite{HJ2017}, and the model 
\begin{eqnarray}\label{multieqin3}
  \begin{split}
  \epsilon 
d\vect{E}-\nabla\times\vect{H}dt=-\vect{J}_e(t,\vect{x},\vect{E},\vect{H})dt-\vect{J}_e^r
  (t,\vect{x},\vect{E},\vect{H})dW, \\
  \mu 
  d\vect{H}+\nabla\times\vect{E}dt=-\vect{J}_m(t,\vect{x},\vect{E},\vect{H})dt-\vect{J}_m^r
  t,\vect{x},\vect{E},\vect{H})dW,
  \end{split}
\end{eqnarray}
which was studied in \cite{CHJ2019}. 
The properties of stochastic Maxwell equations are also provided in that paper. 
Hong et al. in \cite{HH2021} studied the stochastic wave equation
% in the form
% \begin{equation}
%  \begin{cases}
%du=vdt, ~~~~~~~~~~~~~~~~~~~~~~~~~~~~~~~(x,t)\in (a,b)\times (0,T]\\
%dv=\Delta udt-f(u)dt+g(u)dW_t.~~(x,t)\in (a,b)\times (0,T]
%\end{cases}
%\end{equation}
and developed numerical schemes that preserve the averaged energy evolution 
law. Both the compact finite difference method and the interior penalty discontinuous Galerkin (DG) finite element 
methods were proposed. Finite element approximations of a class of nonlinear stochastic wave equations with multiplicative noise were recently investigated in \cite{LWX2021}.

The high order DG finite element methods are considered in this paper. 
The DG method is a class of finite element methods that uses discontinuous piecewise polynomials as the basis functions, and was first introduced by Reed and Hill in \cite{RH1973} to solve linear transport equation. In the early 1990s, Cockburn et al. studied the extension of DG methods for hyperbolic conservation laws in a series of papers \cite{CHS1990,  CKS2000,  CLS1989, CS1989}. DG methods adopt many advantages from both finite element and finite volume methods, including hp-adaptivity flexibility, efficient parallel implementation, the ability of handling complicated boundary conditions, etc., making them a popular choice for numerical methods of conservation laws. 
  
In \cite{CCLX2017}, Cheng et al. studied DG methods for the one-dimensional (1D) deterministic two-way wave equations,
%  \begin{eqnarray}\label{multicheng2017}
%    \begin{cases}
%      E_t=B_x-S_1,\\
%      B_t=E_x-S_2,
%    \end{cases}
%  \end{eqnarray} 
%  where $E, B$ are unknowns and $S_1, S_2$ are source terms. 
and investigated a family of $L^2$ stable high order DG methods defined through a general form of numerical fluxes. A systematic study of stability, error estimates, and dispersion analysis was carried out. In a recent paper \cite{SX2021}, Sun and Xing extended the analysis on optimal error estimate to multi-dimensional wave equations. For the DG methods with generalized numerical fluxes, one key idea to the optimal error estimate was by constructing a global projection on unstructured meshes, which will also be useful for the analysis in this paper. 
Recently, DG methods have been extended for stochastic partial differential equations. Li et al. investigated DG methods \cite{YL 2020} for stochastic conservation laws with multiplicative noise
% \[du+f(u)_xdt=g(x,t,u)dW_t, \]
and provided optimal error estimate for the semilinear equations.

In this paper, we present the DG methods for both one- and two-dimensional (2D) stochastic Maxwell equations with multiplicative noise. This is an extension of our previous work \cite{SXS2021}, where we studied multi-symplectic DG methods for stochastic Maxwell equations with additive noise. We showed in \cite{SXS2021} that the proposed methods satisfy the discrete form of the stochastic energy linear growth property and preserve the multi-symplectic structure on the discrete level. Optimal error estimate of the semi-discrete DG method was also analyzed in \cite{SXS2021}. Another related work was studied by Chen in \cite{C}, where a symplectic full discretization of multi-dimensional stochastic Maxwell equations with additive noise is provided. DG methods were used for spatial discretization and midpoint method was used for temporal discretization. A first order convergence in time and 1.5th order convergence in space was discussed in \cite{C}.  
In this work, we plan to apply DG spatial discretization with generalized numerical fluxes to stochastic Maxwell equations with multiplicative noise. We will present the discrete energy growth property of the numerical solutions obtained from our semi-discrete DG scheme, following the energy law of the exact solutions.
With the help of a global projection constructed in \cite{SX2021}, we present optimal error estimate of the semi-discrete DG methods for stochastic Maxwell equations with multiplicative noise for both one- and two-dimensional cases. In the two-dimensional case, we study both rectangular meshes and triangular meshes. Strong Taylor 2.0 temporal discretization is combined with the semi-discrete DG method to derive a fully discrete method for numerical implementation. Both one- and two-dimensional numerical results are presented to validate the theoretical analysis results.
  
The structure of this paper is as follows. Section \ref{multi1D} describes the DG method for one-dimensional stochastic Maxwell equations with multiplicative noises. Energy growth and the optimal error estimate of the proposed method are provided.  Section \ref{multi2D} studies the DG scheme for two-dimensional stochastic Maxwell equations. Both rectangular meshes and triangular meshes are considered. The corresponding discrete energy law and optimal error estimate are also studied for both cases. The temporal discretization is briefly discussed in Section \ref{multitime}. Section \ref{multinum} presents the numerical results to validate the theoretical results. Conclusion remarks are provided in Section \ref{multisec6}.
  
%Throughout this paper, $L^2$ norm is denoted by $\norm{\cdot}$, and $C$ represents a generic 
%positive constant independent of the spatial and temporal step size $h$ and $\Delta t$, 
%which can take different values in different cases. Furthermore, $W_t$ represents 
%the standard Brownian motion starting from 0. 

\section{One-dimensional stochastic Maxwell equations with multiplicative noise}\label{multi1D}
\setcounter{equation}{0}\setcounter{figure}{0}\setcounter{table}{0}

In this section, we consider one-dimensional stochastic Maxwell equations with multiplicative noise:
  \begin{equation}\label{multi1}
  \begin{cases}
dv=-u_xdt+f(x,t,u,v)dW_t, &\\
du=-v_xdt+g(x,t,u,v)dW_t,
\end{cases}
\end{equation}
where $W_t$ is a standard one-dimensional Brownian motion on a given probability space $(\Omega,\mathcal{F},\mathbb{P})$
and $f,~g$ are functions that satisfy the following Lipschitz continuous and linear growth assumptions:
\begin{eqnarray}
  && | f(x,t,u_1,v)-f(x,t,u_2,v)|+ | g(x,t,u_1,v)-g(x,t,u_2,v)|\le 
   C|u_1-u_2|, \label{multiassum1}\\
  && |f(x,t,u,v_1)-f(x,t,u,v_2)|+|g(x,t,u,v_1)-g(x,t,u,v_2)|\le C|v_1-v_2|,\label{multiassum2}\\
 & & |f(x,t,u,v)|+|g(x,t,u,v)|\le C(1+|u|+|v|). \label{multiassum3}
\end{eqnarray}

The stochastic Maxwell equations with multiplicative noise \eqref{multi1} satisfy the following energy law. 
A similar result was discussed in \cite{HH2021} for stochastic wave equations.
% \JS{We next present the energy law for \eqref{multi1}. There is a similar result discussed in \cite{HH2021} for stochastic wave equations}
  \begin{theorem}[\bf{Continuous energy law}]\label{multi1denergypdelevel}
Let $u$ and $v$ be the solutions to the model \eqref{multi1} under periodic boundary condition. For any $t$, the global stochastic energy satisfies the following energy law
\begin{eqnarray}\label{multithm2.1}
     \mathbb{E}\Big(\int_I u^2(x,t)+v^2(x,t)dx\Big) = \int_I u^2(x,0)+v^2(x,0)dx  +\int_0^t \mathbb{E}(\norm{f}^2+\norm{g}^2) 
     d\tau.
  \end{eqnarray}
  \end{theorem}
%  \YX{Is there any reference to this result? Add it, and may skip the proof if it is available elsewhere.} 
  \begin{proof}
      By utilizing the It\^{o}'s lemma and equations \eqref{multi1}, we have 
  \begin{align}
          &d\Big(\int_I u^2(x,t)+v^2(x,t)dx\Big)=\int_I \big(2udu+ d\langle u,u\rangle_t+2v dv +d\langle v,v\rangle_t\big) dx \notag\\
    &~~~~~~~=-\int_I (2uv_x+2v u_x)dtdx+2\int_I (gu+fv) dW_tdx+\int_I d\langle u,u\rangle_t+d\langle v,v\rangle_t dx \notag\\
    &~~~~~~~=2\int_I (gu+fv) dW_tdx+\int_I d\langle u,u\rangle_t+d\langle v,v\rangle_t dx,
        \label{multi1denergyeq1}
  \end{align}
where the last equality follows from the integration by parts and periodic boundary condition.
Integrating over time leads to
 \begin{align}
          \int_I u^2(x,t)+v^2(x,t)dx&=\int_I u^2(x,0)+v^2(x,0)dx   \notag\\
    & + 2\int_0^t \int_I (gu+fv)dx dW_t +\int_I \langle u,u\rangle_t+\langle v,v\rangle_t dx.	\label{multi1denergyeq2}
  \end{align}  
  
Integrating the second equation of \eqref{multi1} over $t$ leads to
  \begin{eqnarray}\label{quadrau}
    u(x,t)=u(x,0)+\int_0^t -v_xd\tau+\int_0^t g(x,\tau,u,v)dW_\tau,
  \end{eqnarray}
therefore, we have
  \begin{eqnarray*}
    \int_I\langle u,u\rangle_tdx= \int_I \Big\langle \int_0^t g(x,\tau,u,v)dW_\tau, \int_0^t g(x,\tau,u,v)dW_\tau\Big\rangle_t dx.
  \end{eqnarray*}
By It\^{o} isometry, we obtain
  \begin{eqnarray*}
    \mathbb{E} \int_I \langle u,u\rangle_t dx=\int_0^t \mathbb{E}\norm{g}^2d\tau.
  \end{eqnarray*}
For the same reason, it follows that 
    \begin{eqnarray*}
    \mathbb{E} \int_I\langle \eta,\eta\rangle_t dx = \int_0^t \mathbb{E}\norm{f}^2d\tau.
  \end{eqnarray*}
%  \YX{Need some more explanations, like why $d\langle u,u\rangle_t$ becomes $\int_I g^2 dxdt$.}
Note that the process $\int_0^t \int_I (gu+fv)dx dW_t$ is an It\^{o} integral, thus it has zero expectation.
Taking the expectation of \eqref{multi1denergyeq2} yields the continuous energy law \eqref{multithm2.1}.
  \end{proof}

%\YX{Talk more about regularity? See how this was discussed in \cite{LZ2019} or \cite{CHJ2019}.}
In \cite{CHJ2019}, Chen et al. established the regularity properties of the solution of stochastic Maxwell equations with multiplicative noise \eqref{multieqin3}. Let $\mathbb{M}$ be the differential operator defined as 
\[\mathbb{M}=\left(\begin{array}{cc} 0&\nabla\times \\ -\nabla\times&0 \end{array}\right).\]
It was shown in \cite{CHJ2019} that, under certain assumptions, the solutions are uniformly bounded in the following way: 
\begin{eqnarray}\label{multiregu1}
  \norm{\mathbb{U}(t)}_{L^p(\Omega;\mathcal{D}(M^l))}\le C\left(1+\norm{\mathbb{U}_0}_{L^p(\Omega;\mathcal{D}(M^l))}\right),
\end{eqnarray}
for any given integer $l$, $p\ge 2$, $t\in[0,T]$, where $\mathbb{U}=(\vect{E},\vect{H})^T$ and $\mathcal{D}(\mathbb{M}^l)$ stands for the domain of $\mathbb{M}^l$, the $l$-th power of the operator $\mathbb{M}$.
Furthermore, the H\"older continuity of the solution holds in $\mathcal{D}(\mathbb{M}^{l-1})$ norm in the expectation sense: 
\begin{eqnarray}\label{multiregu2}
  \mathbb{E}\norm{\mathbb{U}(t)-\mathbb{U}(s)}^p_{\mathcal{D}(\mathbb{M}^{l-1})}\le 
  |t-s|^{\frac p2}.
\end{eqnarray}
%We will assume $H^{k+1}$ regularity in this paper  to study the ``best'' spatial convergence rate of the proposed method, which has also been observed on some numerical examples in Section \ref{multinum}. 

The one-dimensional computational domain $I$ is partitioned into subintervals $I_j=[x_{j-1/2},x_{j+1/2}]$, $j=1,2,\cdots,N$. Let us denote 
$x_j=(x_{j-1/2}+x_{j+1/2})/2$ to be the center of each subinterval. Denote $h_j=x_{j+1/2}-x_{j-1/2}$ to be the mesh size of each subinterval and $h=\max_j h_j$ to be the maximum mesh size. We further assume that the ratio $h/h_j$ is bounded over all $j$ during mesh refinement. 
The piecewise polynomial solution and test function space $V_h^k$ is defined as
\[V_h^k=\{w_h:w_h|_{I_j}\in P^k(I_j),~~ j=1,2,\cdots,N\},\]
where $P^k(I_j)$ stands for the space of polynomials of degree up to $k$ on the cell $I_j$.
Since the function in $V_h^k$ can be discontinuous at cell interface, we use $w_h^+$ and $w_h^-$ to represent the function limit from the right and left respectively. 
We denote the average and jump of the functions at the cell interfaces by $\{w_h\} = \left(w_h^++w_h^-\right)/2$ and $[w_h] = w_h^+ - w_h^-$.
Throughout this paper, we denote $L^2$ norm by $\norm{\cdot}$, and $C$ represents a generic 
positive constant which is independent of the spatial and temporal step size $h$, $\Delta t$, and
may take different values in different cases.

The DG scheme for the one-dimensional model \eqref{multi1} takes the following formulation: 
for $x\in I,~ (\omega,t)\in \Omega\times [0,T]$, find $v_h(\omega,x,t), u_h(\omega,x,t) \in V^k_h$, such that for any test functions $ \varphi, \widetilde{\varphi}\in V_h^k$, it holds that 
\begin{eqnarray}
  \int_{I_j}dv_h\varphi(x)dx=\Big(\int_{I_j}u_h\varphi_xdx 
  -(\widehat{u}_h\varphi^-)_{j+\frac12}+(\widehat{u}_h\varphi^+)_{j-\frac12}\Big)dt
  +\int_{I_j}f\varphi dW_tdx,   \label{multiphi}\\
    \int_{I_j}du_h\widetilde{\varphi}(x)dx=\Big(\int_{I_j}v_h\widetilde{\varphi}_xdx 
  -(\widehat{v}_h\widetilde{\varphi}^-)_{j+\frac12}+(\widehat{v}_h\widetilde{\varphi}^+)_{j-\frac12}\Big)dt
  +\int_{I_j}g\widetilde{\varphi} dW_tdx,   \label{multiphii}
  \end{eqnarray}
where $\widehat{u}_h, \widehat{v}_h$ are the numerical fluxes defined on the cell interfaces. In this paper, we follow the study in \cite{CCLX2017,SX2021} and consider the following generalized numerical fluxes
\begin{equation}\label{multiflux1}
\widehat{u}_h=\{u_h\}+\alpha[u_h]-\beta_1[v_h], \qquad \widehat{v}_h=\{v_h\}-\alpha[v_h]-\beta_2[u_h],
\end{equation}
for some $\alpha\in\mathbb{R}$ and $\beta_1,~\beta_2\geq 0$. A subgroup of such numerical fluxes, named ``$\alpha\beta$'' fluxes, were considered in \cite{CCLX2017} for one-dimensional deterministic two-way wave equations, and optimal error estimate were investigated based on a specially constructed projection operator.
 DG method with more general numerical fluxes were studied in \cite{SX2021} for the one- and multi-dimensional deterministic wave equations. To provide an optimal error estimate, the key ingredient was to construct an optimal global projection on one-dimensional meshes or multi-dimensional unstructured meshes, which will also be utilized in this paper. Note that similar DG method was studied in \cite{C} for the stochastic Maxwell equation with additive noise, using a specific choice of numerical fluxes (upwind flux) and slightly different treatment of the noise term, and suboptimal error estimate was investigated.  

Next, we start by showing the following semi-discrete energy law satisfied by the numerical solutions of the proposed DG methods.  
\begin{theorem}[\bf{Semi-discrete energy law}]\label{multithm1}
  Let $v_h$ and $u_h$ be the numerical solutions obtained in \eqref{multiphi} and 
  \eqref{multiphii} with $\beta_{1},\beta_2\geq 0$, and let $\mathcal{P}(f)$ and $\mathcal{P}(g)$ be the $L^2$ projections of $f$ and $g$ onto $V_h^k$, then we have
  \[\mathbb{E}\Big(\norm{u_h(x,t)}^2+\norm{v_h(x,t)}^2\Big)\le\norm{u_h(x,0)}^2+\norm{v_h(x,0)}^2+  \int_0^t\mathbb{E}\Big(\norm{\mathcal{P}(f)}^2+\norm{\mathcal{P}(g)}^2\Big)ds. \]
Moreover, the equality holds when $\beta_1=\beta_2=0$ in the numerical fluxes \eqref{multiflux1}.
\end{theorem}
\begin{proof}
Taking the test functions $\varphi=v_h$ in \eqref{multiphi} and $\widetilde{\varphi}=u_h$ in \eqref{multiphii}, and adding the resulting two equations together, we have   
  \begin{eqnarray}\label{multiqqq}
     \int_{I_j}(du_h)u_h+(dv_h)v_h dx&=&
\int_{I_j}gu_h+fv_hdxdW_t
+ \Big(\int_{I_j}u_h(v_h)_xdx+v_h(u_h)_xdx\Big)dt \notag \\
&+&\Big( \big((\{u_h\}+\alpha[u_h]-\beta_1[v_h])v_h^+\big)_{j-\frac12}
  -\big((\{u_h\}+\alpha[u_h]-\beta_1[v_h])v_h^-\big)_{j+\frac12}\Big)dt
\notag \\
&+&\Big( \big((\{v_h\}-\alpha[v_h]-\beta_2[u_h])u_h^+\big)_{j-\frac12}
  -\big((\{v_h\}-\alpha[v_h]-\beta_2[u_h])u_h^-\big)_{j+\frac12}\Big)dt\notag \\
 &=&\int_{I_j}gu_h+fv_hdxdW_t+(\Theta_{j-\frac12}-\Theta_{j+\frac12})dt-(\beta_1[v_h]^2_{j+\frac12}+\beta_2[u_h]^2_{j+\frac12})dt\notag\\
 &\le&\int_{I_j}gu_h+fv_hdxdW_t + (\Theta_{j-\frac12}-\Theta_{j+\frac12})dt,
  \end{eqnarray}
where \[\Theta=\left(\frac12+\alpha\right)v^-_hu_h^++\left(\frac12-\alpha\right)u_h^-v_h^+.\]
Note that the equality in \eqref{multiqqq} holds when $\beta_1=\beta_2=0$ in the numerical fluxes \eqref{multiflux1}.

By It\^{o}'s formula, we have
 \begin{equation}\label{multiIto}
    d(u_h)^2=2u_hdu_h+d\langle u_h,u_h\rangle_t, \qquad d(v_h)^2=2v_hdv_h+d\langle v_h,v_n\rangle_t.
 \end{equation}
 It follows from \eqref{multiphi} that 
 \begin{align}
  \int_{I_j}v_h(x,t)\varphi(x)dx&=\int_{I_j}v_h(x,0)\varphi(x)dx 	  \label{multiphi1} \\ 
  &
  +\int_0^t\Big(\int_{I_j}u_h\varphi_xdx -(\widehat{u}_h\varphi^-)_{j+\frac12}+(\widehat{u}_h\varphi^+)_{j-\frac12}\Big)d\tau
  +\int_0^t\int_{I_j}f\varphi dW_{\tau}dx, \notag 
 \end{align}
 therefore, for any continuous semimartingale $Y$, 
 \begin{equation} \label{multiITOO}
\int_{I_j}\langle v_h,Y\rangle_t\varphi(x)dx=\Big\langle\int_{I_j}v_h\varphi dx,Y\Big\rangle_t
=\Big\langle\int_0^t\int_{I_j}f\varphi dW_{\tau}dx,Y\Big\rangle_t=\Big\langle\int_0^t\int_{I_j}\mathcal{P}(f)\varphi dW_{\tau}dx,Y\Big\rangle_t.
 \end{equation}
Let us represent the numerical solutions $v_h$ in the cell $I_j$ as
\[v_h(\omega,x,t)=\sum^k_{l=0}v_j^l(\omega,t)\phi_j^l(x),\]
where $\{\phi_j^l\}$ represents the set of orthogonal Legendre basis of $V_h^k$ over cell $I_j$.
It can be shown that \eqref{multiphi1} and \eqref{multiITOO} lead to
 \begin{eqnarray}\label{multiquadr}
   \begin{split}
     \int_{I_j}\langle v_h,v_h\rangle_tdx=\int_{I_j}\Big\langle v_h,\sum_{l=0}^k
     v_j^l\phi_j^l\Big\rangle_tdx=\sum_{l=0}^k\int_{I_j}\langle 
     v_h,v_j^l\rangle_t\phi_j^ldx=\sum_{l=0}^k
     \Big\langle\int_0^t\int_{I_j}\mathcal{P}(f)\phi_j^ldW_{\tau}dx,v_j^l\Big\rangle_t
     \\
     =\sum_{l=0}^k\int_{I_j}\Big\langle\int_0^t\mathcal{P}(f)dW_{\tau},v_j^l\phi_j^l\Big\rangle_tdx
     =\int_{I_j}\Big\langle\int_0^t\mathcal{P}(f)dW_{\tau},v_h\Big\rangle_tdx,
   \end{split}
 \end{eqnarray}
Let us denote $\mathcal{P}(f)=\sum_{l=0}^k ({P_f})_j^l\phi_j^l$. Repeating the same process as in \eqref{multiquadr}, we can obtain 
 \begin{eqnarray}\label{multiquadr2}
   \int_{I_j}\Big\langle\int_0^t\mathcal{P}(f)dW_{\tau},v_h\Big\rangle_tdx=\int_{I_j}
   \Big\langle\int_0^t\mathcal{P}(f)dW_{\tau}, \int_0^t\mathcal{P}(f)dW_{\tau}\Big\rangle_tdx,
 \end{eqnarray}
which leads to 
 \begin{eqnarray}\label{multimaeta}
   \int_{I_j}\langle v_h,v_h\rangle_tdx=\int_{I_j}
   \Big\langle\int_0^t\mathcal{P}(f)dW_{\tau}, \int_0^t\mathcal{P}(f)dW_{\tau}\Big\rangle_tdx
   =:m_a(v_h,f).
 \end{eqnarray}
 Similarly, we have 
  \begin{eqnarray}\label{multimau}
   \int_{I_j}\langle u_h,u_h\rangle_tdx=\int_{I_j}
   \Big\langle\int_0^t\mathcal{P}(g)dW_{\tau}, \int_0^t\mathcal{P}(g)dW_{\tau}\Big\rangle_tdx
   =:m_a(u_h,g).
 \end{eqnarray}
 
% In \eqref{multiphi} choose $\varphi=dv_h$, we have 
% \[\int_{I_j} (dv_h)^2dx=\int_{I_j} fdv_h dxdW_t=\int_{I_j}(\mathcal{P}f)dv_hdxdW_t,\]
% where $\mathcal{P}f$ is the $L^2$ projection of $f$ onto $V_h^k$.
% Again,  in \eqref{multiphi} choose $\varphi=(\mathcal{P}f)dW_t$, we have 
% \[\int_{I_j}(\mathcal{P}f)dv_hdxdW_t=\int_{I_j}(\mathcal{P}f)^2dxdt,\]
%which leads to 
%\begin{eqnarray}\label{multidetah}
%  \int_{I_j} (dv_h)^2dx=\int_{I_j}(\mathcal{P}f)^2dxdt.
%\end{eqnarray}
%
% 
% Similarly, we have 
% \begin{eqnarray}\label{multiduh}
%\int_{I_j} (du_h)^2dx=\int_{I_j}(\mathcal{P}g)^2dxdt.
% \end{eqnarray}
Integrating \eqref{multiqqq} over $t$, summing up over all cells $I_j$ and utilizing the results in \eqref{multiIto}, \eqref{multimaeta}-\eqref{multimau},  we have 
\begin{eqnarray}
  \norm{u_h(x,t)}^2+\norm{v_h(x,t)}^2&\le&  \norm{u_h(x,0)}^2+\norm{v_h(x,0)}^2 \notag \\
  &+&m_a(v_h;f)+m_a(u_h;g) +\int_0^t\int_I gu_h+fv_h dxdW_{\tau}, \label{multienergy1}
\end{eqnarray}
with periodic boundary conditions.
Note that the process $\int_0^t\int_I gu_h+fv_h dxdW_{\tau}$ is an It\^{o} 
integral, thus it has zero expectation.
Taking expectation of equation \eqref{multienergy1}, and applying It\^{o} isometry onto $m_a(v_h;f)$ and $m_a(u_h;g)$, we have 
\begin{eqnarray*}
  \mathbb{E}\Big(\norm{u_h(x,t)}^2+\norm{v_h(x,t)}^2\Big)
  \le \norm{u_h(x,0)}^2+\norm{v_h(x,0)}^2+
  \int_0^t\mathbb{E}\Big(\norm{\mathcal{P}(f)}^2+\norm{\mathcal{P}(g)}^2\Big)d\tau,
\end{eqnarray*}
which finishes the proof.
%From an application of Gronwall's inequality we finally obtain 
%\[ \mathbb{E}\Big(\norm{u_h(x,t)}^2+\norm{v_h(x,t)}^2\Big)\le (C+ \norm{u_h(x,0)}^2+\norm{v_h(x,0)}^2)
%e^{Ct}.\]
\end{proof}

Next, we will provide an optimal error estimate analysis of the proposed semi-discrete DG method.
We start by defining a pair of global projection operators which will be used in the error 
estimate analysis: on any cell $I_j$ and for any pair of smooth functions $(q(x),~r(x))$, 
define $\mathcal{P}^{\alpha,\beta_1}$ and $\mathcal{P}^{-\alpha,\beta_2}$ as 
\begin{eqnarray}
  &&\int_{I_j}(\mathcal{P}^{\alpha,\beta_1} q-q(x))w(x)dx=0,\qquad \forall w(x)\in 
  P^{k-1}(I_j),\quad \forall j,\label{multip1}\\
   &&\int_{I_j}(\mathcal{P}^{-\alpha,\beta_2} r-r(x))w(x)dx=0,\qquad \forall w(x)\in 
  P^{k-1}(I_j), \quad \forall j,\\
  &&(\{\mathcal{P}^{\alpha,\beta_1}   q\}+\alpha[\mathcal{P}^{\alpha,\beta_1}   
  q]-\beta_1[\mathcal{P}^{\alpha,\beta_1}r])_{j+\frac12}=q(x_{j+\frac12}), \quad \forall j,\\
  &&(\{\mathcal{P}^{-\alpha,\beta_2}   r\}-\alpha[\mathcal{P}^{-\alpha,\beta_2}   
  r]-\beta_2[\mathcal{P}^{-\alpha,\beta_2}q])_{j+\frac12}=r(x_{j+\frac12}),\quad \forall j.\label{multip2}
\end{eqnarray}
This set of global projections was also introduced by Sun and Xing in \cite{SX2021}, 
and the following property on the projection error is studied in 
\cite[Lemma 2.1]{SX2021} and will be useful in the error estimate analysis.

\begin{lemma}[\bf{Projection error}] \label{multiprojectionerror}
  If $\alpha^2+\beta_1\beta_2\ne0$, the projections $\mathcal{P}^{\alpha,\beta_1},~\mathcal{P}^{-\alpha,\beta_2}$ in \eqref{multip1}-\eqref{multip2} are well defined.
   Furthermore, let $q,~r\in H^{k+1}$ be smooth functions, then there exists some constant $C$ such that 
  \[\norm{\mathcal{P}^{\alpha,\beta_1} q-q}^2+\norm{\mathcal{P}^{-\alpha,\beta_2} r-r}^2\le Ch^{2k+2}\left(\norm{q}^2_{H^{k+1}}+\norm{r}^2_{H^{k+1}}\right).\]
\end{lemma}

\begin{theorem}[\bf{Optimal error estimate}]\label{multierror}
Suppose $\beta_1,\beta_2\geq 0$ and $\alpha^2+\beta_1\beta_2\ne0$. Let $u_h$ and $v_h$ be the numerical solutions obtained by the semi-discrete DG method 
\eqref{multiphi}-\eqref{multiphii}, and $u,~v \in L^2(\Omega\times[0,T],H^{k+2})\cap L^4(\Omega\times [0,T];\mathbb{R})$ be the strong solutions to \eqref{multi1} with 
$f(\cdot,\cdot,u(\cdot),v(\cdot)), g(\cdot,\cdot,u(\cdot),v(\cdot))\in L^2(\Omega\times[0,T],H^{k+1})$, then there 
  exists some constant $C$ such that  
%  \YX{The space of $u,v$ are more restrictive in [17]. Try to find out whether what we write here is sufficient.]}
  \begin{eqnarray}\label{multiesti1}
      \mathbb{E}\Big(\norm{u-u_h}^2+\norm{v-v_h}^2\Big)\le Ch^{2k+2}.
  \end{eqnarray}
  \end{theorem}
  \begin{proof} 
     Since the exact solutions $v$ and $u$ also satisfy equations \eqref{multiphi}-\eqref{multiphii}, taking the difference of the semi-discrete method and the equations satisfied by the exact solutions yields the following error equations:
   \begin{eqnarray}
  \int_{I_j}d(v-v_h)\varphi(x)dx=\Big(\int_{I_j}(u-u_h)\varphi_xdx 
  -((u-\widehat{u}_h)\varphi^-)_{j+\frac12}+((u-\widehat{u}_h)\varphi^+)_{j-\frac12}\Big)dt\notag 
  \\
  +\int_{I_j} (f(x,t,u,v)-f(x,t,u_h,v_h))\varphi dxdW_t,\label{multiphie}\\
    \int_{I_j}d(u-u_h)\widetilde{\varphi}(x)dx=\Big(\int_{I_j}(v-v_h)\widetilde{\varphi}_xdx 
  -((v-\widehat{v}_h)\widetilde{\varphi}^-)_{j+\frac12}+((v-\widehat{v}_h)\widetilde{\varphi}^+)_{j-\frac12}\Big)dt
  \notag \\
  +\int_{I_j} (g(x,t,u,v)-g(x,t,u_h,v_h))\widetilde{\varphi} dxdW_t.\label{multiphiie}
  \end{eqnarray}
Decompose the numerical error into the following two terms
  \begin{equation}\label{multiooo}
   v-v_h=\xi^v-\epsilon^v,\qquad u-u_h=\xi^u-\epsilon^u,
  \end{equation} 
where
  \begin{eqnarray*}
    \xi^v=\mathcal{P}^{-\alpha,\beta_2}v-v_h  \in V_h^k,~~~\epsilon^v=\mathcal{P}^{-\alpha,\beta_2}v-v,~~~
    \xi^u=\mathcal{P}^{\alpha,\beta_1}   u-u_h  \in V_h^k,~~~\epsilon^u=\mathcal{P}^{\alpha,\beta_1}   u-u.
  \end{eqnarray*}
For the initial condition, we choose
\[v_h(x,0)=\mathcal{P}^{-\alpha,\beta_2}v(x,0), \qquad u_h(x,0)=\mathcal{P}^{\alpha,\beta_1} u(x,0), \]
hence $\xi^v(x,0)=\xi^u(x,0)=0$. For $q\in\{f,g\}$ and any $w\in V_h^k$, define 
   \[\mathcal{E}^j_q(w)=\int_{I_j}\big(q(x,t,u,v)-q(x,t,u_h,v_h)\big)wdxdW_t, \]
which is an It\^{o} integral, hence we have $\mathbb{E}\Big(\mathcal{E}^j_q(w)\Big)=0$.
By choosing the test functions $\varphi=\xi^v$, $\widetilde{\varphi}=\xi^u$ in \eqref{multiphie}-\eqref{multiphiie} and summing up the equations, we obtain
  \begin{align}
&          \int_{I_j}d\xi^v\xi^v+d\xi^u\xi^u dx=\int_{I_j}d\epsilon^v\xi^v+d\epsilon^u\xi^u dx
          +\mathcal{E}^j_f(\xi^v)+\mathcal{E}^j_g(\xi^u)\notag \\       
    &\hskip18mm
    +\Big(\int_{I_j}\xi^u\xi^v_xdx
    -\big((\{\xi^u\}+\alpha[\xi^u]-\beta_1[\xi^v])(\xi^v)^-\big)_{j+\frac12}
    +\big((\{\xi^u\}+\alpha[\xi^u]-\beta_1[\xi^v])(\xi^v)^+\big)_{j-\frac12}\Big)dt \notag \\
    &\hskip18mm
    -\Big(\int_{I_j}\epsilon^u\xi^v_xdx
     -\big((\{\epsilon^u\}+\alpha[\epsilon^u]-\beta_1[\epsilon^v])(\xi^v)^-\big)_{j+\frac12}
    +\big((\{\epsilon^u\}+\alpha[\epsilon^u]-\beta_1[\epsilon^v])(\xi^v)^+\big)_{j-\frac12}\Big)dt\notag \\
    &\hskip18mm
    +\Big(\int_{I_j}\xi^v\xi^u_xdx
    -\big((\{\xi^v\}-\alpha[\xi^v]-\beta_2[\xi^u])(\xi^u)^-\big)_{j+\frac12}
    +\big((\{\xi^v\}-\alpha[\xi^v]-\beta_2[\xi^u])(\xi^u)^+\big)_{j-\frac12}\Big)dt \notag \\ 
    &\hskip18mm
    -\Big(\int_{I_j}\epsilon^v\xi^u_xdx
     -\big((\{\epsilon^v\}-\alpha[\epsilon^v]-\beta_2[\epsilon^u])(\xi^u)^-\big)_{j+\frac12}
    +\big((\{\epsilon^v\}-\alpha[\epsilon^v]-\beta_2[\epsilon^u])(\xi^u)^+\big)_{j-\frac12}\Big)dt\notag \\
    &\hskip18mm
    = \int_{I_j}d\epsilon^v\xi^v+d\epsilon^u\xi^u dx+ 
    \Big(\widetilde{\Theta}_{j-\frac12}-\widetilde{\Theta}_{j+\frac12}\Big)dt
    +\mathcal{E}^j_f(\xi^v)+\mathcal{E}^j_g(\xi^u)-(\beta_1[\xi^v]^2_{j+\frac12}+\beta_2[\xi^u]^2_{j+\frac12})dt,\notag \\
    &\hskip18mm
    \le \int_{I_j}d\epsilon^v\xi^v+d\epsilon^u\xi^u dx+ 
    \Big(\widetilde{\Theta}_{j-\frac12}-\widetilde{\Theta}_{j+\frac12}\Big)dt
    +\mathcal{E}^j_f(\xi^v)+\mathcal{E}^j_g(\xi^u),	\label{multierr3}
  \end{align}
  where 
  $\widetilde{\Theta}=\big(\frac12+\alpha\big)(\xi^u)^+(\xi^v)^-+\big(\frac12-\alpha\big)(\xi^v)^+(\xi^u)^-$.
The last equality follows from the definition of the special projection which leads to 
(for the error term $\epsilon^v=\mathcal{P}^{-\alpha,\beta_2}v-v$, $\epsilon^u=\mathcal{P}^{\alpha,\beta_1} u-u$)    
  \[\int_{I_j}\epsilon^u\xi^v_xdx=\int_{I_j}\epsilon^v\xi^u_xdx=(\{\epsilon^u\}+\alpha[\epsilon^u]-\beta_1[\epsilon^v])_{j\pm\frac12}
  =(\{\epsilon^v\}-\alpha[\epsilon^v]-\beta_2[\epsilon^u])_{j\pm\frac12}=0,\]
and an integration by parts which leads to
\begin{eqnarray*}
 && \int_{I_j}\xi^u\xi^v_xdx+\int_{I_j}\xi^v\xi^u_xdx-\big((\{\xi^u\}+\alpha[\xi^u]-\beta_1[\xi^v])(\xi^v)^-\big)_{j+\frac12}
    +\big((\{\xi^u\}+\alpha[\xi^u]-\beta_1[\xi^v])(\xi^v)^+\big)_{j-\frac12}\\
     &&~~~~~~~~~~~~~~~~~~~~~~~~~~~~~~~-\big((\{\xi^v\}-\alpha[\xi^v]-\beta_2[\xi^u])(\xi^u)^-\big)_{j+\frac12}
    +\big((\{\xi^v\}-\alpha[\xi^v]-\beta_2[\xi^u])(\xi^u)^+\big)_{j-\frac12}\\
    &&~~~~~~~~~~~~~~~~~~~~~~~~~~~~~~~~=\widetilde{\Theta}_{j-\frac12}-\widetilde{\Theta}_{j+\frac12}
    -(\beta_1[\xi^v]^2_{j+\frac12}+\beta_2[\xi^u]^2_{j+\frac12}).
\end{eqnarray*}
    
By It\^{o}'s lemma, we have
\begin{eqnarray}\label{multiito2}
 d(\xi^v)^2=2d\xi^v\xi^v+d\langle \xi^v,\xi^v\rangle_t,~
 d(\xi^u)^2=2d\xi^u\xi^u+d\langle \xi^u,\xi^u\rangle_t.
\end{eqnarray}
Note that 
\[d(\mathcal{P}^{\alpha,\beta_1}  u)=\mathcal{P}^{\alpha,\beta_1}  (du)=\mathcal{P}^{\alpha,\beta_1}  (-v_xdt+gdW_t)=
\mathcal{P}^{\alpha,\beta_1}  (-v_xdt)+\mathcal{P}^{\alpha,\beta_1}  (g)dW_t.\]
We then have, for any test function $\widetilde{\varphi}$, 
\begin{eqnarray}\label{multiproj}
  \int_{I_j}(d\mathcal{P}^{\alpha,\beta_1}  u)\widetilde{\varphi}dx
  =\int_{I_j}\mathcal{P}^{\alpha,\beta_1}  (-v_x)\widetilde{\varphi}dxdt+\int_{I_j}\widetilde{\varphi}
  \mathcal{P}^{\alpha,\beta_1}  (g)dxdW_t.\label{multianafopro1}
\end{eqnarray}
Subtracting \eqref{multiphii} from \eqref{multiproj}, we have 
\begin{eqnarray}\label{multixiues}
  \begin{split}
    \int_{I_j}d\xi^u\widetilde{\varphi}dx=\Big(\int_{I_j}(-v_h\widetilde{\varphi_x}+\mathcal{P}^{\alpha,\beta_1}  
  (-v_x)\widetilde{\varphi})
  dx+(\widehat{v_h}\widetilde{\varphi}^-)_{j+\frac12}-(\widehat{v_h}\widetilde{\varphi}^+)_{j-\frac12}
  \Big)dt\\
  +\int_{I_j}(\mathcal{P}^{\alpha,\beta_1}  
  (g(\cdot,u,v))-g(\cdot,u_h,v_h))\widetilde{\varphi}dxdW_t.
  \end{split}
\end{eqnarray}
Therefore, after integrating over $t$, we have 
\begin{eqnarray}
  \begin{split}
   \int_{I_j}\xi^u\widetilde{\varphi}dx=\int_0^t\Big(\int_{I_j}(-v_h\widetilde{\varphi_x}+\mathcal{P}^{\alpha,\beta_1}  
  (-v_x)\widetilde{\varphi})
  dx+(\widehat{v_h}\widetilde{\varphi}^-)_{j+\frac12}-(\widehat{v_h}\widetilde{\varphi}^+)_{j-\frac12}
  \Big)d\tau\\
  +\int_0^t\int_{I_j}(\mathcal{P}^{\alpha,\beta_1}  
  (g(\cdot,u,v))-g(\cdot,u_h,v_h))\widetilde{\varphi}dxdW_{\tau}.
  \end{split}
\end{eqnarray}
For any semimartingale $Y$, we have 
\begin{eqnarray}
  \int_{I_j}\langle\xi^u,Y\rangle_t \widetilde{\varphi}dx=\Big\langle\int_0^t\int_{I_j}(\mathcal{P}^{\alpha,\beta_1}  
  (g(\cdot,u,v))-g(\cdot,u_h,v_h))\widetilde{\varphi}dxdW_{\tau},Y\Big\rangle_t.
\end{eqnarray}
Let us write $\xi^u=\sum_{l=0}^k (\xi^u)_j^l\phi_j^l$. After repeating the same process as in \eqref{multiquadr} and applying the Cauchy inequality, we obtain
\begin{eqnarray} \label{thm2.3eq1}
  \begin{split}
     &\hskip-5mm
   \int_{I_j}\langle\xi^u,\xi^u\rangle_tdx=\int_{I_j}\Big\langle\int_0^t \mathcal{P}^{\alpha,\beta_1}  
  (g(\cdot,u,v))-g(\cdot,u_h,v_h)dW_{\tau},\xi^u\Big\rangle_tdx\\
  &=\int_{I_j}\Big\langle\int_0^t \mathcal{P}^{\alpha,\beta_1}  
  (g(\cdot,u,v))-g(\cdot,u,v)dW_{\tau}, \xi^u\Big\rangle_t+\Big\langle\int_0^t
  g(\cdot,u,v)-g(\cdot,u_h,v_h)dW_{\tau}, \xi^u\Big\rangle_tdx\\
  &\le \frac12 \int_{I_j}\langle\xi^u,\xi^u\rangle_tdx+
  C\int_{I_j}\Big\langle\int_0^t \mathcal{P}^{\alpha,\beta_1}  
  (g)-g dW_{\tau}, \int_0^t\mathcal{P}^{\alpha,\beta_1}  
  (g)-g dW_{\tau}\Big\rangle_tdx\\
  &+C\int_{I_j}\Big\langle \int_0^t
  g(\cdot,u,v)-g(\cdot,u_h,v_h)dW_{\tau}, \int_0^t
  g(\cdot,u,v)-g(\cdot,u_h,v_h)dW_{\tau}\Big\rangle_tdx.\\
  \end{split}
\end{eqnarray}
Introduce the notations
\begin{eqnarray*}
  &&Qu_1=\int_{I_j}\Big\langle\int_0^t \mathcal{P}^{\alpha,\beta_1}  
  (g)-g dW_{\tau}, \int_0^t\mathcal{P}^{\alpha,\beta_1}  
  (g)-g dW_{\tau}\Big\rangle_tdx,\\
   &&Qu_2=\int_{I_j}\Big\langle \int_0^t
  g(\cdot,u,v)-g(\cdot,u_h,v_h)dW_{\tau}, \int_0^t
  g(\cdot,u,v)-g(\cdot,u_h,v_h)dW_{\tau}\Big\rangle_tdx.
\end{eqnarray*}
%Let $\widetilde{\varphi}=d\xi^u$, and by an application of Young's inequality we then obtain
%\begin{eqnarray}
%  \begin{split}
%    \int_{I_j}(d\xi^u)^2dx&=\int_{I_j}(\mathcal{P}^{\alpha,\beta_1}  (g(\cdot,u,v))-g(\cdot,u_h,v_h))d\xi^u dxdW_t\\
%    &=\int_{I_j}(\mathcal{P}^{\alpha,\beta_1}  (g(\cdot,u,v))-g(\cdot,u,v))d\xi^u dxdW_t
%    +\int_{I_j}((g(\cdot,u,v))-g(\cdot,u_h,v_h))d\xi^u dxdW_t\\
%    &\le \frac12 \int_{I_j}
%(d\xi^u)^2dx+C \int_{I_j}(\mathcal{P}^{\alpha,\beta_1}  g-g)^2dxdt 
%+C    \int_{I_j}|g(\cdot,u,v)-g(\cdot,u_h,v_h)|^2dxdt,\label{multianforpro}
%  \end{split}
%\end{eqnarray}
Recall the assumptions \eqref{multiassum1} and \eqref{multiassum2}, and by It\^{o} isometry we have 
\begin{eqnarray}\label{multiguh}
  \begin{split}
    \mathbb{E}(Qu_2)=\int_{I_j}\int_0^t\mathbb{E}|g(\cdot,u,v)-g(\cdot,u_h,v_h)|^2d\tau dx&\le 
    C\int_{I_j}\int_0^t\mathbb{E}(|u-u_h|^2+|v-v_h|^2 )d\tau dx\\
    &=C\int_{I_j}\int_0^t\mathbb{E}(|\xi^u-\epsilon^u|^2+
    |\xi^v-\epsilon^v|^2)d\tau dx,
  \end{split}
\end{eqnarray}
and 
\[\mathbb{E}(Qu_1)=\int_{I_j}\int_0^t\mathbb{E}(\mathcal{P}^{\alpha,\beta_1}  g-g)^2d\tau dx.\]
Therefore, after taking the expectation of equation \eqref{thm2.3eq1}, we conclude 
\begin{eqnarray}\label{multidxiu}
  \int_{I_j}\mathbb{E}(\langle\xi^u,\xi^u\rangle_t) dx\le C\int_{I_j}\int_0^t\mathbb{E}(\mathcal{P}^{\alpha,\beta_1}  
  g-g)^2d\tau dx
  +C\int_{I_j}\int_0^t\mathbb{E}(|\xi^u-\epsilon^u|^2+
    |\xi^v-\epsilon^v|^2)d\tau dx,
\end{eqnarray}
and for the same reason, 
\begin{eqnarray}\label{multidxieta}
  \int_{I_j}\mathbb{E}(\langle \xi^v,\xi^v\rangle_t)dx\le C\int_{I_j}\int_0^t\mathbb{E}(\mathcal{P}^{-\alpha,\beta_2}f-f)^2d\tau dx+C\int_{I_j}\int_0^t\mathbb{E}(|\xi^u-\epsilon^u|^2+
    |\xi^v-\epsilon^v|^2)d\tau dx.
\end{eqnarray}

After summing over all cells $I_j$, utilizing the periodic boundary conditions and integrating equations \eqref{multierr3}, \eqref{multiito2} from $0$ to $t$, we can take the expectation of the resulting equations and obtain
\begin{eqnarray*}
  \mathbb{E}\Big(\norm{\xi^v(x,t)}^2+\norm{\xi^v(x,t)}^2\Big)&\le&
    \norm{\xi^v(x,0)}^2+\norm{\xi^v(x,0)}^2+
  \mathbb{E}\Big(\int^t_0\int_I d\epsilon^v\xi^v+d\epsilon^u\xi^udxd\tau\Big)  \\  
    &+&  \int_0^t\int_{I}\mathbb{E}(\langle \xi^u,\xi^u\rangle_t)dxd\tau
    +  \int_0^t\int_{I}\mathbb{E}(\langle \xi^v,\xi^v\rangle_t)dxd\tau.
\end{eqnarray*}
%  \begin{align}
%& \mathbb{E} \Big(\int_{I}  d(\xi^v)^2 +  d(\xi^u)^2 dx \Big)
%    \le \mathbb{E}\Big( \int_{I}d\epsilon^v\xi^v+d\epsilon^u\xi^u dx \Big)
%    +  \int_0^t\int_{I}\mathbb{E}(\langle \xi^u,\xi^u\rangle_t)dxd\tau
%    +  \int_0^t\int_{I}\mathbb{E}(\langle \xi^v,\xi^v\rangle_t)dxd\tau,	\label{thm2.3eq2}
%  \end{align}
%  where 
%  
Following the same derivation of Eq. \eqref{multixiues}, we have 
\begin{eqnarray}
   \begin{split}
    \int_{I}d\epsilon^u\widetilde{\varphi}dx=\int_{I}\big(v_x-\mathcal{P}^{\alpha,\beta_1}  
  (v_x)\big)\widetilde{\varphi}
  dxdt
  +\int_{I}(\mathcal{P}^{\alpha,\beta_1}  
  (g)-g)\widetilde{\varphi}dxdW_t.
  \end{split}
\end{eqnarray}
Based on the assumptions, we know that $\int_{I_j}(\mathcal{P}^{\alpha,\beta_1}  
  (g)-g)\widetilde{\varphi}dxdW_t$ is a martingale. Therefore 
  \begin{eqnarray}\label{multidepsilones1}
    \mathbb{E}\Big(\int_0^t\int_Id\epsilon^u\xi^u dx\Big)=\mathbb{E}\Big(\int_0^t\int_{I}\big(v_x-\mathcal{P}^{\alpha,\beta_1}  
  (v_x)\big)\xi^u
  dxd\tau\Big)\le Ch^{2k+2}+\int_0^t\mathbb{E}\norm{\xi^u(x,\tau)}^2d\tau,
  \end{eqnarray}
after applying the error estimate of the projection in Lemma \ref{multiprojectionerror}.  
  Similarly, we have
  \begin{eqnarray}\label{multidepsilones2}
    \mathbb{E}\Big(\int_0^t\int_Id\epsilon^v\xi^v dx\Big)\le Ch^{2k+2}+\int_0^t\mathbb{E}\norm{\xi^v(x,\tau)}^2d\tau.
  \end{eqnarray}
Note that the initial condition satisfies $\norm{\xi^v(x,0)}=\norm{\xi^u(x,0)}=0$, and we can utilize the results in \eqref{multidxiu}-\eqref{multidepsilones2} to obtain
%\YX{Do we need $g,f\in H^{k+1}$ space for the projection error to be bounded?}
\begin{align*}
  &\mathbb{E}\Big(\norm{\xi^v(x,t)}^2+\norm{\xi^v(x,t)}^2\Big)  \\
  & \hskip1cm
  \le\mathbb{E}\Big(\int^t_0\int_I d\epsilon^v\xi^v+d\epsilon^u\xi^udx\Big) 
  +\mathbb{E}\Big(\int_0^t\norm{\mathcal{P}^{\alpha,\beta_1}  g-g}^2+\norm{\mathcal{P}^{-\alpha,\beta_2}f-f}^2d\tau \Big)\\
  & \hskip1cm
  +\mathbb{E}\Big(\int_0^t \norm{\xi^v(x,\tau)}^2+\norm{\xi^u(x,\tau)}^2 d\tau \Big)+Ch^{2k+2}\\
  & \hskip1cm
  \le C\int_0^t \mathbb{E}(\norm{\xi^v(x,\tau)}^2+\norm{\xi^u(x,\tau)}^2)d\tau+Ch^{2k+2}.
\end{align*}  
The optimal error estimate \eqref{multiesti1} follows from applying the Gronwall's inequality and the optimal projection error. 
  \end{proof} 
  \begin{remark}\label{rmk1} %\YX{To be revised}
In the proof, we assumed enough regularity of the exact solutions to study the ``best'' spatial convergence rate of the proposed method. Such convergence rate has also been observed on some numerical examples in Section \ref{multinum}. 
%In the literature, $H^2$ regularity for solutions to stochastic Maxwell equations is often assumed, and in 
%\cite{CHJ2019}
%the authors show that for any given integer $k$, the solution is uniformly bounded in $\mathcal{D}(M^k)$ norm if $u_0$ has bounded $\mathcal{D}(M^k)$ norm, where $\norm{u}_{\mathcal{D}(M^k)}=(\norm{u}^2+\norm{M^ku}^2)^{\frac12}$ with the operator $M$ defined by 
%\[M=\left(\begin{array}{cc} 0&\nabla\times \\ -\nabla\times&0 \end{array}\right).\]
%Same remark also for Theorem \ref{multi2derror} and Theorem \ref{multitrierror}
  \end{remark}

\section{Two-dimensional stochastic Maxwell equations with multiplicative noise}\label{multi2D}
\setcounter{equation}{0}\setcounter{figure}{0}\setcounter{table}{0}

In this section, we study two-dimensional stochastic Maxwell equations in the following form
 \begin{equation}\label{multi2d}
  \begin{cases}
dE-T_xdt+S_ydt=f(x,t,\bm{u})dW_t,\\
dS+E_ydt=g(x,t,\bm{u})dW_t,\\
dT-E_xdt=r(x,t,\bm{u})dW_t,
\end{cases}
\end{equation}
where $\bm{u}=(E,S,T)$,  and $f,~g,~r$ are 
smooth functions that satisfy the following Lipschitz continuous and linear growth assumptions:
\begin{align}
 &| f(x,t,\bm{u}_1)-f(x,t,\bm{u}_2)|+| g(x,t,\bm{u}_1)-g(x,t,\bm{u}_2)|+|r(x,t,\bm{u}_1)-r(x,t,\bm{u}_2)|\le 
 C|\bm{u}_1-\bm{u}_2|, \label{multi2dassum1}\\
 &|f(x,t,\bm{u})|+|g(x,t,\bm{u})|+|r(x,t,\bm{u})|\le C(1+|E|+|S|+|T|).\label{multi2dassum2}
\end{align}
%\YX{Talk more about regularity?}
We refer to the beginning of Section \ref{multi1D} on the discussion of regularity properties of the solutions.
%JS{Similar to one dimensional case, the solutions to \eqref{multi2d} satisfy 
%the regularity results \eqref{multiregu1} and \eqref{multiregu2} in two 
%dimensional case, and here we assume $H^{k+1}$ regularity for all solutions.}
For this model, we have the following energy law satisfied by the exact solutions. The proof follows almost the same analysis as that of Theorem \ref{multi1denergypdelevel} and is skipped here.
\begin{theorem}[\bf{Continuous energy law}]\label{multi2denergypdelevel}
  Let $E,~S,~T$ be the solutions to the qeuation \eqref{multi2d} on the bounded domain $\Omega$ with the periodic boundary condition, then for any $t$, the global stochastic energy satisfies the following
energy law
%\begin{eqnarray}\label{multi2dEL}
% \mathcal{E}(t)=\mathcal{E}(0)+2\int_0^t\int_J\int_I \lambda_2(S+T)-\lambda_1 E dW_{\tau}dxdy+(\lambda_1^2+2\lambda_2^2)Tr(Q)t,
%\end{eqnarray}
  \begin{align}\label{multilinearpdelevel}
    &\mathbb{E}\Big(\int_\Omega E(x,y,t)^2+S(x,y,t)^2+T(x,y,t)^2 dxdy\Big) \notag \\
     &\hskip10mm
     =\int_\Omega E(x,y,0)^2+S(x,y,0)^2+T(x,y,0)^2 dxdy+\int_0^t \mathbb{E}(\norm{f}^2+\norm{g}^2+\norm{r}^2)d\tau.
  \end{align}
\end{theorem}

%%%%%%%%%%%%%%%%%%%%%%%%%%%%%%%%%%%%%%%%%%%
\subsection{Triangular meshes}
In this subsection we will consider DG methods for stochastic Maxwell equations with multiplicative noise on triangular discretization of the domain $\Omega$, and carry out the corresponding analysis.
Firstly we rewrite the equation \eqref{multi2d} in the following compact form: 
\begin{eqnarray}\label{multi2dtri}
  dE=\nabla\cdot \vect{U}dt+{F}(\vect{x},t,E,\vect{U})dW_t, ~~~d\vect{U}=\nabla 
  Edt+\vect{G}(\vect{x},t,E,\vect{U})dW_t,
\end{eqnarray}
where $\vect{U}=(T,-S)^T$, and $F$ and $\vect{G}$ are 
functions satisfying \eqref{multi2dassum1} and \eqref{multi2dassum2}.

Let $\Omega= \cup_{K\in\mathcal{T}_h} K$ be a quasi-uniform triangulation of 
the domain $\Omega$, and $\Gamma$ be the collection of triangle faces. For a 
triangle $K$, and a face $\mathscr{F}\in \partial K$, let $\vect{n}$ be the outer normal 
vector on $\mathscr{F}$.  Given a face $\mathscr{F}$, let $\vect{n}_\mathscr{F}$ denote the unit vector across $\mathscr{F}$, whose 
direction is not essential for unspecified $K$, and we further assume there is a unit vector $\vect{r}$ such that $|\vect{r}\cdot\vect{n}_\mathscr{F}|\ge\kappa>0$ for all $\vect{n}_\mathscr{F}$.
 We define the finite element DG space $V_h^k$ as
\[V_h^k=\{v_h:v_h|_K\in P^k(K)\},\]
where $P^k$ is the space of polynomial of degree at most $k$. $\mathbb{V}_h$ is used to denote $(V_h^k)^2$, and denote $(P^k)^2$ as $\vect{P}^k$.
For a given face $\mathscr{F}$, let $K^{\pm}$ be the two neighboring cells. For any 
function $w$ or vector valued function $\vect{v}$, define 
\[w^{\pm}=\lim_{\vect{x}\to\partial K, \vect{x}\in K^{\pm}}w(\vect{x}), ~~~~~~~\vect{v}^\pm = \lim_{\vect{x}\to\partial K, \vect{x}\in K^{\pm}}\vect{v}(\vect{x}). \]
The vector $\vect{n}^{\pm}$ is the outer normal vector on $F$ with respect to $K^\pm$. 
Finally, we define the following notations for averages and jumps across a face $\mathscr{F}$:
\begin{eqnarray}
 && \{w\}=\frac12 (w^++w^-),~~~~~~~\{\vect{v}\}=\frac12(\vect{v}^++\vect{v}^-),\\
 && [w\vect{n}] =w^+\vect{n}^++w^-\vect{n}^-,~~~~~~[\vect{v}\cdot\vect{n}]=\vect{v}^+\cdot\vect{n}^++\vect{v}^-\cdot\vect{n}^-.
\end{eqnarray}
Note that over all the triangles and faces, we have the following equalities
\begin{eqnarray}\label{multiintprop}
  \sum_{K\in\mathcal{T}_h}\int_{\partial 
  K}w\vect{v}\cdot\vect{n}ds=\sum_{\mathscr{F}\in\Gamma}\int_Fw^+\vect{n}^+\cdot\vect{v}^++
  w^-\vect{n}^-\cdot\vect{v}^-ds=\sum_{\mathscr{F}\in\Gamma}\int_F 
  \{w\}[\vect{n}\cdot\vect{v}]+[w\vect{n}]\cdot\{\vect{v}\}ds.
\end{eqnarray}

The DG scheme for the two-dimensional system \eqref{multi2dtri} is: find $E_h\in V^k_h,~ 
\vect{U}_h\in\mathbb{V}_h^k$, such that for all test functions $\varphi\in V_h^k,~~\bm{\psi}\in\mathbb{V}_h^k$, it holds that
\begin{eqnarray}
  &&\int_KdE_h\varphi d\vect{x}+\int_K \vect{U}_h\cdot\nabla\varphi d\vect{x}dt-\int_{\partial K}\bm{\mathcal{F}_1}(\vect{U}_h,E_h)\cdot 
  \varphi\vect{n}dsdt=\int_K F\varphi d\vect{x}dW_t, \label{multitridg1}\\
  &&\int_{K}d\vect{U}_h\cdot 
  \bm{\psi}d\vect{x}+\int_{K}E_h\nabla\cdot\bm{\psi}d\vect{x}dt-\int_{\partial 
  K}\mathcal{F}_2(E_h,\vect{U}_h)\bm{\psi}\cdot\vect{n}dsdt=\int_K\vect{G}\cdot\bm{\psi}d\vect{x}dW_t,\label{multitridg2}
\end{eqnarray}
where the numerical fluxes are chosen to be 
\begin{equation}\label{multiflux2}
\bm{\mathcal{F}_1}(\vect{U}_h,E_h)=\{\vect{U}_h\}-\bm{\alpha}[\vect{U}_h\cdot\vect{n}]-\beta_1[E_h\vect{n}],~~~
\mathcal{F}_2(E_h,\vect{U}_h)=\{E_h\}+\bm{\alpha}\cdot[E_h\vect{n}]-\beta_2[\vect{U}_h\cdot\vect{n}],
\end{equation}
with $\bm{\alpha}=\alpha \sign(\vect{r}\cdot\vect{n}_\mathscr{F})\vect{n}_\mathscr{F}$ for some number $\alpha$. %,  and $\beta_1\ge0,~\beta_2> 0$.
Note that these generalized numerical fluxes can be viewed as two-dimensional extension of the one-dimensional numerical fluxes \eqref{multiflux1}.

Next, we start by showing the following semi-discrete energy law satisfied by numerical solutions of the proposed DG methods.
\begin{theorem}[\textbf{Semi-discrete energy law}]\label{multitrieng}
  Let $E_h$ and $\vect{U}_h$ be the numerical solutions obtained in \eqref{multitridg1} and 
  \eqref{multitridg2} with $\beta_1\ge0,~\beta_2\geq 0$, then we have
  \begin{eqnarray}
    \begin{split}
          \mathbb{E}\Big(\norm{E_h(\vect{x},t)}^2+\norm{\vect{U}_h(\vect{x},t)}^2\Big)\le\norm{E_h(\vect{x},0)}^2+\norm{\vect{U}_h(\vect{x},0)}^2
              +\int_0^t\mathbb{E}\Big(\norm{\mathcal{P}(F)}^2+\norm{\mathcal{P}(\vect{G})}^2\Big)d\tau.
    \end{split}
  \end{eqnarray}
Moreover, the equality holds when $\beta_1=\beta_2=0$ in the numerical fluxes \eqref{multiflux2}.  
\end{theorem}

\begin{proof}
By taking the test function $\varphi=E_h$ in \eqref{multitridg1} and $\bm{\psi}=\vect{U}_h$ in \eqref{multitridg2}, and summing the resulting equations over all cells $K$, we obtain 
  \begin{eqnarray}
    \begin{split}
          &\int_{\Omega} 
    dE_hE_h+d\vect{U}_h\cdot\vect{U}_hd\vect{x}+\int_{\Omega}\vect{U}_h\cdot\nabla{E_h}
    +E_h\nabla\cdot\vect{U}_hd\vect{x}dt\\
    &-\sum_{K\in\mathcal{T}_h}\int_{\partial K}
    \bm{\mathcal{F}_1}(\vect{U}_h,E_h)\cdot 
    E_h\vect{n}+\mathcal{F}_2(E_h,\vect{U}_h)\vect{U}_h\cdot\vect{n}dsdt
    =\int_{\Omega}FE_h+\vect{G}\cdot\vect{U}_h
    d\vect{x}dW_t.\label{multitrienergy1}
    \end{split}
  \end{eqnarray}
  By an integration by parts and applying the equality \eqref{multiintprop}, we can have 
  \begin{eqnarray}
    \begin{split}
     & \int_{\Omega}\vect{U}_h\cdot\nabla{E_h}
    +E_h\nabla\cdot\vect{U}_hd\vect{x}dt-\sum_{K\in\mathcal{T}_h}\int_{\partial K}
    \bm{\mathcal{F}_1}(\vect{U}_h,E_h)\cdot 
    E_h\vect{n}+\mathcal{F}_2(E_h,\vect{U}_h)\vect{U}_h\cdot\vect{n}dsdt\\
    &=\sum_{K\in\mathcal{T}_h}\int_{\partial K}E_h\vect{U}_h\cdot\vect{n}dsdt
    -\sum_{\mathscr{F}\in\Gamma}\int_F \bm{\mathcal{F}_1}(\vect{U}_h,E_h)\cdot[E_h\vect{n}]+\mathcal{F}_2(E_h,\vect{U}_h)[\vect{U}_h\cdot\vect{n}]
    dsdt\\
    &=\sum_{K\in\mathcal{T}_h}\int_{\partial K}E_h\vect{U}_h\cdot\vect{n}dsdt
    -\sum_{\mathscr{F}\in\Gamma}\int_F\Big(\{\vect{U}_h\}-\bm{\alpha}[\vect{U}_h\cdot\vect{n}]\Big)\cdot[E_h\vect{n}]
    +\Big(\{E_h\}+\bm{\alpha}\cdot[E_h\vect{n}]\Big)[\vect{U}_h\cdot\vect{n}]dsdt\\
    &+\sum_{\mathscr{F}\in\Gamma}\int_F \beta_1|[E_h\vect{n}]|^2+\beta_2[\vect{U}_h\cdot 
    \vect{n}]^2dsdt= \sum_{\mathscr{F}\in\Gamma}\int_F \beta_1|[E_h\vect{n}]|^2+\beta_2[\vect{U}_h\cdot 
    \vect{n}]^2dsdt.\label{multianaener}
    \end{split}
  \end{eqnarray}
Therefore, equation \eqref{multitrienergy1} becomes
  \begin{eqnarray}
    \begin{split}
       \int_{\Omega} 
    dE_hE_h+d\vect{U}_h\cdot\vect{U}_hd\vect{x}+\sum_{\mathscr{F}\in\Gamma}\int_F \beta_1|[E_h\vect{n}]|^2+\beta_2[\vect{U}_h\cdot 
    \vect{n}]^2dsdt=\int_{\Omega}FE_h+\vect{G}\cdot\vect{U}_h d\vect{x}dW_t.\label{multitrienergy2}
    \end{split}
  \end{eqnarray}
  
By It\^{o}'s lemma, we have
  \begin{eqnarray*}
    dE_hE_h=\frac12 (d(E_h)^2-d\langle E_h,E_h\rangle_t), 
    ~~d\vect{U}_h\cdot\vect{U}_h=\frac12(d|\vect{U}_h|^2-d\langle \vect{U}_h,\vect{U}_h\rangle_t).
  \end{eqnarray*}  
Following an exact same process as in the derivation of \eqref{multimaeta}, and applying the It\^{o} isometry, we have 
  \begin{eqnarray*}
    \int_K\mathbb{E}\langle E_h,E_h\rangle_td\vect{x}=\int_K\int_0^t 
    \mathbb{E}(\mathcal{P}(F))^2d\tau d\vect{x}, \qquad \int_K\mathbb{E}\langle 
    \vect{U}_h,\vect{U}_h\rangle_td\vect{x}=\int_K\int_0^t\mathbb{E}(\mathcal{P}(\vect{G}))^2d\tau d\vect{x}.
  \end{eqnarray*}
%  In \eqref{multitridg1} choose $\varphi=dE_h$ and $\varphi=\mathbb{P}FdW_t$, we have 
%  \[\int_{K}(dE_h)^2 d\vect{x}= \int_K (\mathbb{P}F)^2d\vect{x}dt,  \]
%  Similarly, we have 
%  \[\int_{K}|d\vect{U}_h|^2 d\vect{x}=  \int_K |\mathbb{P}\vect{U}|^2d\vect{x}dt.\]
By plugging these into \eqref{multitrienergy2}, integrating over time $t$, summing over all $K$, and taking expectation, we obtain
%  \begin{eqnarray}
%    \begin{split}
%      \int_{\Omega}\mathbb{E}((E_h)^2+|\vect{U}_h|^2)d\vect{x}\le 
%       \int_\Omega\mathbb{E}\langle E_h,E_h\rangle_td\vect{x}+\int_K\mathbb{E}\langle 
%    \vect{U}_h,\vect{U}_h\rangle_td\vect{x}
%      = \int_{\Omega}\int_0^t\mathbb{E}\Big(\norm{\mathcal{P}(F)}^2+\norm{\mathcal{P}(\vect{G})}^2\Big)d\tau d\vect{x},\label{multitrienergy4}
%    \end{split}
%  \end{eqnarray}
  \begin{eqnarray}
    \mathbb{E}\Big(\norm{E(\vect{x},t)}^2+\norm{\vect{U}(\vect{x},t)}^2\Big)
    \le \norm{E(\vect{x},0)}^2+\norm{\vect{U}(\vect{x},0)}^2
    +\int_0^t\mathbb{E}\Big(\norm{\mathcal{P}(F)}^2+\norm{\mathcal{P}(\vect{G})}^2\Big)d\tau,
  \end{eqnarray}
where we use the fact that the right-hand side of \eqref{multitrienergy1} $\int_{\Omega}FE_h+\vect{G}\cdot\vect{U}_h d\vect{x}dW_t$ is a martingale, hence its expectation equals to zero. It is easy to observe that the equality holds when $\beta_1=\beta_2=0$.  
\end{proof}
Similar to the one-dimensional case, to provide the optimal error estimate, we need to introduce the following pair of projections $\mathbb{P}^{\vect{U}}\vect{U}$ and $\mathcal{P}^E E$ \cite{SX2021}: for any $K\in\mathcal{T}_h$,
\begin{eqnarray}
 && \int_K (\mathbb{P}^{\vect{U}}\vect{U}-\vect{U})\cdot \vect{v}d\vect{x}=0,
  \quad \forall\vect{v}\in\vect{P}^{k-1}(K), \forall K\in \mathcal{T}_h, \label{multipro1}\\
  && \int_K(\mathcal{P}^EE-E)wd\vect{x}=0,
  \quad \forall w\in P^{k-1}(K), \forall K\in \mathcal{T}_h,  \label{multipro2}\\
  &&\int_{\partial K}\bm{\mathcal{F}_1}(\mathbb{P}^{\vect{U}}\vect{U},\mathcal{P}^E E)\cdot 
  \mu\vect{n}ds=\int_{\partial K}\vect{U}\cdot\mu\vect{n}ds, 
  \quad \forall\mu\in P^k(\mathscr{F}), \forall\mathscr{F}\in\Gamma, \label{multipro3}\\
 && \int_{\partial K}\mathcal{F}_2(\mathcal{P}^E E, \mathbb{P}^{\vect{U}}\vect{U})\nu ds=\int_{\partial K}E\nu ds, 
  \quad \forall\nu\in P^k(\mathscr{F}), \forall\mathscr{F}\in\Gamma.\label{multipro4}
\end{eqnarray}
%\YX{what does $V_h^k(\partial K)$ and $V_h^k(K)$ mean? They are not defined. See how it was done in reference [16].}
The following lemma on the projection is studied in \cite[Lemma 3.1]{SX2021} and will be useful 
in the analysis of error estimate:
\begin{lemma}\label{multiSXlemma}
  Suppose $\beta_1\ge0,~\beta_2>0$, and $|\bm{\alpha}|^2+\beta_1\beta_2\ne0$, then 
  the projection pair defined in \eqref{multipro1} - \eqref{multipro4} is well defined,
 and there exists some constant $C$ independent of mesh size $h$, such that 
  \[\norm{\mathbb{P}^\vect{U}\vect{U}-\vect{U}}^2+\norm{\mathcal{P}^EE-E}^2\le Ch^{2k+2}(\norm{\vect{U}}^2_{H^{k+1}}+
  \norm{E}^2_{H^{k+1}}).\]
\end{lemma}
\begin{theorem}[\textbf{Optimal error estimate}]\label{multitrierror}
     Suppose $\beta_1\ge0,~\beta_2>0$, and $|\bm{\alpha}|^2+\beta_1\beta_2\ne0$.
Let $E_h$ and $\vect{U}_h$ be the numerical solutions obtained by semi-discrete DG method \eqref{multitridg1}-\eqref{multitridg2}, 
and $E, \vect{U} \in L^2(\Omega\times[0,T];H^{k+2})\cap L^4(\Omega\times [0,T];\mathbb{R})$ are strong solutions, and $F, \vect{G}\in L^2(\Omega\times[0,T],H^{k+1})$, then there exists some constant $C$ such that 
  \[\mathbb{E}\Big(\norm{E-E_h}^2+\norm{\vect{U}-\vect{U}_h}^2\Big)\le Ch^{2k+2}.\]
\end{theorem}
\begin{proof}
  For simplicity, for $f,~\varphi\in V_h^k$, and $\vect{g},~\bm{\psi}\in \mathbb{V}_h^k$, introduce the notation
  \begin{eqnarray*}
    &&\mathcal{A}_K(f,\vect{g};\varphi,\bm{\psi})=\int_K\vect{g}\cdot\nabla\varphi+f\nabla\cdot\bm{\psi}
     d\vect{x}dt-
    \int_{\partial K} 
    \bm{\mathcal{F}_1}(\vect{g},f)\cdot\varphi\vect{n}+\mathcal{F}_2(f,\vect{g})\bm{\psi}\cdot\vect{n}dsdt.
  \end{eqnarray*}
  We further define 
  \[\xi^\vect{U}=\mathbb{P}^{\vect{U}}\vect{U}-\vect{U}_h,\quad 
  \xi^E=\mathcal{P}^EE-E_h, \quad  
  \epsilon^{\vect{U}}=\mathbb{P}^{\vect{U}}\vect{U}-\vect{U}, \quad
  \epsilon^E=\mathcal{P}^EE-E, \]
and choose the initial condition
\[\vect{U}_h(x,0)=\mathbb{P}^{\vect{U}}\vect{U}(x,0), \qquad E_h(x,0)=\mathcal{P}^EE(x,0), \]
hence $ \xi^E(\vect{x},0)=0,~ \xi^\vect{U}(\vect{x},0)=\vect{0}$.

It can be observed that the following error equations hold
\begin{align*}
  &\int_K d(\xi^E-\epsilon^E)\varphi d\vect{x}+\int_K (\xi^\vect{U}-\epsilon^{\vect{U}})\cdot\nabla\varphi d\vect{x}dt
  -\int_{\partial K}\bm{\mathcal{F}_1}(\xi^\vect{U}-\epsilon^{\vect{U}},\xi^E-\epsilon^E)\cdot \varphi\vect{n}dsdt \\
  & \hskip3cm
  =\int_K (F(E,\vect{U})-F(E_h,\vect{U}_h))\varphi d\vect{x}dW_t, \\
  &\int_{K}d(\xi^\vect{U}-\epsilon^{\vect{U}})\cdot 
  \bm{\psi}d\vect{x}+\int_{K}(\xi^E-\epsilon^E)\nabla\cdot\bm{\psi}d\vect{x}dt
  -\int_{\partial K}\mathcal{F}_2(\xi^E-\epsilon^E,\xi^\vect{U}-\epsilon^{\vect{U}})\bm{\psi}\cdot\vect{n}dsdt\\
  & \hskip3cm
  =\int_K (\vect{G}(E,\vect{U})-\vect{G}(E_h,\vect{U}_h))\cdot\bm{\psi}d\vect{x}dW_t,
\end{align*}
where we dropped the dependence of $F$, $\vect{G}$ on $\vect{x},t$ for simplicity. 
Taking the test functions $\varphi=\xi^E$ and $\bm{\psi}=\xi^{\vect{U}}$, and summing up two equations, we obtain
  \begin{eqnarray}
    \begin{split}
          \int_K d\xi^E\xi^E+d\xi^{\vect{U}}\cdot \xi^{\vect{U}}d\vect{x}-\int_K d\epsilon^E\xi^E+d\epsilon^{\vect{U}}\cdot\xi^{\vect{U}}d\vect{x}
    +\mathcal{A}_K(\xi^E,\xi^{\vect{U}};\xi^E,\xi^{\vect{U}})-\mathcal{A}_K(\epsilon^E,\epsilon^{\vect{U}};\xi^E,\xi^{\vect{U}})\\   
    =\int_K (F(E,\vect{U})-F(E_h,\vect{U}_h))\xi^E+(\vect{G}(E,\vect{U})-\vect{G}(E_h,\vect{U}_h))\cdot\xi^{\vect{U}}d\vect{x}dW_t,
    \end{split}\label{multianaerr}
  \end{eqnarray}
From our definition of projections \eqref{multipro1}-\eqref{multipro4}, we can conclude that $\mathcal{A}_K(\epsilon^E,\epsilon^{\vect{U}};\xi^E,\xi^{\vect{U}})=0$.
Following the exact same analysis as in \eqref{multianaener}, we can have 
  \[\sum_{K\in\mathcal{T}_h}\mathcal{A}_K(\xi^E,\xi^{\vect{U}};\xi^E,\xi^{\vect{U}})=
  \sum_{\mathscr{F}\in\Gamma}\int_{F}\beta_1^2|[\xi^E\vect{n}]|^2+\beta_2[\xi^{\vect{U}}\cdot\vect{n}]^2dsdt\ge0.\]
In addition, the term $M_1:=\int_K(F(E,\vect{U})-F(E_h,\vect{U}_h))\xi^E+(\vect{G}(E,\vect{U})-\vect{G}(E_h,\vect{U}_h))\cdot\xi^{\vect{U}}d\vect{x}dW_t$ on the right-hand side is an martingale, and its expectation equals to zero.   
Therefore, summing the equation \eqref{multianaerr} over all the cells $K$ leads to
  \begin{eqnarray}
    \int_\Omega d\xi^E\xi^E+d\xi^{\vect{U}}\cdot \xi^{\vect{U}}d\vect{x}\le\int_\Omega d\epsilon^E\xi^E+d\epsilon^{\vect{U}}\cdot
    \xi^{\vect{U}}d\vect{x} + M_1.\label{multitriana3}
  \end{eqnarray}
Following the same derivation of Eqs. \eqref{multidepsilones1}-\eqref{multidepsilones2}, we have 
\begin{eqnarray}
  \mathbb{E}\Big(\int_\Omega \int_0^t d\epsilon^E\xi^E+d\epsilon^{\vect{U}}\cdot
    \xi^{\vect{U}}d\vect{x}\Big)\le Ch^{2k+2}+\int_0^t \mathbb{E}\big(\norm{\xi^E({\vect{x}},\tau)}^2+\norm{\xi^{\vect{U}}({\vect{x}},\tau)}^2\big)d\tau.
\end{eqnarray}
  By It\^{o} lemma, we have
  \[d\xi^E\xi^E=\frac12(d(\xi^E)^2-d\langle \xi^E,\xi^E\rangle_t),~~ d\xi^{\vect{U}}\cdot \xi^{\vect{U}}=\frac12(d|\xi^\vect{U}|^2-d\langle \xi^{\vect{U}}, \xi^{\vect{U}}\rangle_t).\]
  Follow the same process as the steps in \eqref{multianafopro1}-\eqref{multidxiu} in one-dimensional case, we have 
  \begin{eqnarray}
   \begin{split}
      \int_K\mathbb{E}\big(\langle \xi^E,\xi^E\rangle_t+\langle \xi^{\vect{U}}, \xi^{\vect{U}}\rangle_t\big) d\vect{x}
    &\le C\int_K \int_0^t
    \mathbb{E}\big(|\mathcal{P}^EF-F|^2+|\mathbb{P}^{\vect{U}}\vect{G}-\vect{G}|^2\big)d\tau d\vect{x}\\
    &+C
    \int_K\int_0^t \mathbb{E}\big(|\xi^E-\epsilon^E|^2+|\xi^\vect{U}-\epsilon^{\vect{U}}|^2\big)d\tau d\vect{x}.\label{multitriana4}
   \end{split}
  \end{eqnarray}
Summing over all the cells $K$, combining these results with \eqref{multitriana3}, integrating over 
time $t$, and taking expectation, we will have 
\begin{equation*}
  \mathbb{E}\Big(\norm{\xi^E(\vect{x},t)}^2+\norm{\xi^\vect{U}(\vect{x},t)}^2\Big)
  \le \norm{\xi^E(\vect{x},0)}^2+\norm{\xi^\vect{U}(\vect{x},0)}^2+
  C\int^t_0 \mathbb{E}\Big(\norm{\xi^E(\vect{x},\tau)}^2+\norm{\xi^\vect{U}(\vect{x},\tau)}^2\Big)
  d\tau+Ch^{2k+2},
\end{equation*}
after utilizing the optimal projection error in Lemma \ref{multiSXlemma}. 
Note that the chosen initial condition satisfies $\xi^E(x,0)=\xi^\vect{U}(x,0)=0$, we have 
  \[\mathbb{E}\Big(\norm{E-E_h}^2+\norm{\vect{U}-\vect{U}_h}^2\Big)\le Ch^{2k+2},\]
after applying the Gronwall's inequality and the optimal projection error. 
\end{proof}

%%%%%%%%%%%%%%%%%%%%%%%%%%%%%%%%%%%%%%%%%%%%%%%%%%%%%%%%%%%%%%%%%%
\subsection{Rectangular meshes}
In this subsection, we will investigate DG methods on cartesian meshes, and study the stability and error estimate of the proposed methods. 
The stability result and its proof is similar to the case of triangular meshes, but different technique is needed for the proof of optimal error estimate, 
which will be discussed in details. 

The two-dimensional rectangular computational domain $\Omega$ is set to be $I\times J$, and we consider the rectangular partition with the cells denoted by $I_i\times J_j=[x_{i-\frac12},x_{i+\frac12}]\times [y_{j-\frac12},y_{j+\frac12}]$ for $i=1,2,\cdots, N_x$ and $j=1,2\cdots, N_y$. 
Let $x_i=\frac12(x_{i-\frac12}+x_{i+\frac12})$, and $y_j=\frac12(y_{j-\frac12}+y_{j+\frac12})$ . Furthermore, we define the mesh size 
in both directions as $h_{x,i}=x_{i+\frac12}-x_{i-\frac12}$, $h_{y,j}=y_{j+\frac12}-y_{j-\frac12}$, with $h_x=\max_i h_{x,i}$, $h_y=\max_j 
h_{y,j}$ and $h=\max(h_x,h_y)$ being the maximum mesh size. Similar to the 
one-dimensional case, we define the two dimensional piecewise polynomial space $\mathbb{V}_h^k$ as follows: 
 \[\mathbb{V}_h^k=\{w(x,y): w|_{I_i\times J_j}\in Q^k(I_i\times J_j)=P^k(I_i)\otimes P^k(J_j),~~i=1,2,\cdots,N_x; j=1,2,\cdots,N_y\}.\]
 
The DG scheme for \eqref{multi2d} is as follows: find $E_h,~S_h,~T_h\in 
\mathbb{V}_h^k$, such that for all $\varphi,~\psi,~\phi\in\mathbb{V}_h^k$, 
 \begin{eqnarray}
   \int_{J_j}\int_{I_i}dE_h\varphi dxdy&=&-\int_{J_j}\mathcal{A}_{I_i}(T_h,\varphi;\alpha_1)
   dydt 
   +\int_{I_i}\mathcal{A}_{J_j}(S_h,\varphi;-\alpha_2)dxdt \notag \\
   && +\int_{J_j}\int_{I_i}f\varphi dxdydW_t, \label{multi2ddgg1}\\
   \int_{J_j}\int_{I_i}dS_h\psi dxdy &=&\int_{I_i}\mathcal{A}_{J_j}(E_h,\psi;\alpha_2)dxdt
        +\int_{J_j}\int_{I_i}
     g\psi dxdydW_t , \label{multi2ddgg2}\\
     \int_{J_j}\int_{I_i} dT_h \phi dxdy
     &=& -\int_{J_j}\mathcal{A}_{I_i}(E_h,\phi;-\alpha_1)dydt
     +\int_{J_j}\int_{I_i}r
     \phi dxdydW_t, \label{multi2ddgg3} 
 \end{eqnarray} 
where, for simplicity, the following operators are defined: for $\alpha\in\mathbb{R}$, 
$p,~q\in\mathbb{V}_h^k$, 
\[\mathcal{A}_{I_i}(p,q;\alpha)=\int_{I_i}pq_xdx-(\widehat{p_\alpha}q^-)_{i+\frac12,y}+(\widehat{p_\alpha}q^+)_{i-\frac12,y},\]
 \[\mathcal{A}_{J_j}(p,q;\alpha)=\int_{J_j}pq_ydy-(\widehat{p_\alpha}q^-)_{x,j+\frac12}+(\widehat{p_\alpha}q^+)_{x,j-\frac12}.\]
 with the numerical fluxes defined as follows:
\begin{equation}\label{multiflux3}
\text{for}~ q\in\mathbb{V}_h^k~\text{and}~ \alpha\in\{\pm\alpha_1,\pm\alpha_2\}\subset\mathbb{R}, ~\widehat{q_\alpha}=\{q\}+\alpha[q].
\end{equation}
Note that the numerical fluxes \eqref{multiflux3} can be viewed as a special case of \eqref{multiflux1} with $\beta_1=\beta_2=0$. They are chosen such that the optimal error estimate can be easily analyzed.

We first provide the following semi-discrete energy law satisfied by numerical solutions of the proposed DG methods on rectangular meshes.
The proof is identical to that of Theorem \ref{multitrieng}, and is skipped here.
\begin{theorem}[\bf{Semi-discrete energy law}]\label{multi2denergy}
Let $E_h,~S_h,~T_h\in\mathbb{V}_h^k$ be the numerical solutions of the semi-discrete DG methods \eqref{multi2ddgg1}-\eqref{multi2ddgg3}, 
then we have  
  \begin{align}
    &   \mathbb{E}\Big(\norm{E_h(x,y,t)}^2+\norm{S_h(x,y,t)}^2+\norm{T_h(x,y,t)}^2\Big)
       \label{multi2denergyy}\\
    &\qquad 
    =\norm{E_h(x,y,0)}^2+\norm{S_h(x,y,0)}^2+\norm{T_h(x,y,0)}^2
    +\int_0^t \mathbb{E}(\norm{\mathcal{P}(f)}^2+\norm{\mathcal{P}(g)}^2+\norm{\mathcal{P}(r)}^2)d\tau.  \notag
    \end{align}
\end{theorem}

Before we study the error estimate, some preparations on projections are provided. 
Let us define the generalized Radau projection as 
\begin{eqnarray}\label{multi2dproj}
   \mathbb{P}^\alpha_x=\mathcal{P}^{\alpha,0}  _x \otimes \mathcal{P}_y,~\mathbb{P}^\alpha_y=\mathcal{P}_x \otimes \mathcal{P}^{\beta,0}  _y,~
 \mathbb{P}^{\alpha,\beta}= \mathcal{P}^{\alpha,0}  _x \otimes \mathcal{P}^{\beta,0}_y,
\end{eqnarray}
where $\mathcal{P} $ is the standard one-dimensional $L^2$ projection, and $\mathcal{P}^{\alpha,0}_x $, $\mathcal{P}^{\beta,0}_y$ are the one-dimensional generalized Radau projections \eqref{multip1}-\eqref{multip2} defined as follows:
On the cell $I_i, J_j$ and for any function $q(x), r(y)$, the projections $\mathcal{P}^{\alpha,0}_x, \mathcal{P}^{\beta,0}_y$ into the space $V_h^k$ are given by
\begin{eqnarray*} 
\int_{I_i}(\mathcal{P}_x^{\alpha,0} q-q(x))\phi(x)dx=0,~\forall \phi(x)\in P^{k-1}(I_i),~~\text{and} ~
(\{\mathcal{P}_x^{\alpha,0} q\}+\alpha[\mathcal{P}^{\alpha,0} q])_{i+\frac12}=q(x_{i+\frac12}), \\
\int_{J_j}(\mathcal{P}_y^{\beta,0} r-r(y))\psi(y)dy=0,~\forall \psi(y)\in P^{k-1}(J_j),~~\text{and} ~
(\{\mathcal{P}_y^{\beta,0} r\}+\alpha[\mathcal{P}^{\beta,0} r])_{j+\frac12}=r(y_{j+\frac12}).
\end{eqnarray*} 
%\YX{Define the projection here. It will be confusing to refer to \eqref{multip1}-\eqref{multip2}}
The following lemmas are studied in \cite{XM 2016} and will be useful in our error estimate analysis.
\begin{lemma}[\bf{Superconvergence property}]\label{multisuperconvergence}
Let $\mathbb{P}^{\alpha,\beta}$ be the projection defined in \eqref{multi2dproj} with $\alpha,\beta\ne0$. For any function $w(x,y)\in H^{k+1}$, denote $\epsilon=\mathbb{P}^{\alpha,\beta}w-w$. For any $\phi\in \mathbb{V}_h^k$, there exists some constant $C$ such that 
    \[\Bigg|\sum_{i,j}\int_{J_j}\mathcal{A}_{I_i}(\epsilon,\phi,\alpha)dy\Bigg|\le 
    Ch^{k+1}\norm{\phi},~~~
  \Bigg|\sum_{i,j}\int_{I_i}\mathcal{A}_{J_j}(\epsilon,\phi,\beta)dx\Bigg|\le Ch^{k+1}\norm{\phi}.\]
 \end{lemma}
 \begin{lemma}[\bf{Projection error}]
Let $\Pi$ be any projection defined in \eqref{multi2dproj} with $\alpha,\beta\ne0$. For any function $w(x,y)\in H^{k+1}$, there exists some constant $C$ such that  \
   \[\norm{\Pi w-w}\le Ch^{k+1}\norm{w}_{H^{k+1}}.\]
 \end{lemma}
Now we turn to the optimal error estimate.  
\begin{theorem} [\bf{Optimal error estimate}]\label{multi2derror}
Let $E_h, S_h, T_h\in\mathbb{V}_h^k$ be the numerical solutions given by the DG scheme \eqref{multi2ddgg1} - \eqref{multi2ddgg3},
and $E, T, S\in L^2(\Omega\times[0,T];H^{k+2})\cap L^4(\Omega\times [0,T];\mathbb{R})$ are exact solutions to \eqref{multi2d}, and $f, g, r\in H^{k+1}$, then there exists some constant $C$ such that 
\begin{equation}\label{multi2derroreq}
\mathbb{E}\Big(\norm{E-E_h}^2+\norm{S-S_h}^2+\norm{T-T_h}^2\Big)\le Ch^{2k+2}.
\end{equation}
\end{theorem}
\begin{proof}
We start by introducing
\begin{align*}
& \xi^E=\mathbb{P}^{-\alpha_1,\alpha_2}E-E_h, \qquad
 \xi^S=\mathbb{P}_y^{-\alpha_2}S-S_h, \qquad
 \xi^T=\mathbb{P}_x^{\alpha_1}T-T_h, \\
 & \epsilon^E=\mathbb{P}^{-\alpha_1,\alpha_2}E-E, \qquad
 \epsilon^S=\mathbb{P}_y^{-\alpha_2}S-S, \qquad 
 \epsilon^T=\mathbb{P}_x^{\alpha_1}T-T, 
 \end{align*}
which leads to 
\[E-E_h=\xi^E-\epsilon^E,\qquad S-S_h=\xi^S-\epsilon^S,\qquad T-T_h=\xi^T-\epsilon^T.\]
We choose the initial conditions as follows:
\[E_h(x,y,0)=\mathbb{P}^{-\alpha_1,\alpha_2}E(x,y,0), \qquad
 S_h(x,y,0)=\mathbb{P}^{-\alpha_2}_yS(x,y,0), \qquad
 T_h(x,y,0)=\mathbb{P}_x^{\alpha_1}T(x,y,0),\]
hence, $\xi^E(x,y,0)=\xi^S(x,y,0)=\xi^T(x,y,0)=0$.

Since both numerical solutions and exact solutions satisfy equations \eqref{multi2ddgg1}-\eqref{multi2ddgg3}, we have the following error equations
\begin{align}
   \int_{J_j}\int_{I_i}d(E-E_h)\varphi dxdy&=-\int_{J_j}\mathcal{A}_{I_i}(T-T_h,\varphi;\alpha_1)
   dydt 
   +\int_{I_i}\mathcal{A}_{J_j}(S-S_h,\varphi;-\alpha_2)dxdt \notag \\
   &~~
   +\int_{J_j}\int_{I_i}(f(x,t,\bm{u})-f(x,t,\bm{u}_h))\varphi dxdydW_t,
  \label{multi2ddgge1}\\
   \int_{J_j}\int_{I_i}d(S-S_h)\psi dxdy 
   &= \int_{I_i}\mathcal{A}_{J_j}(E-E_h,\psi;\alpha_2)dxdt \notag \\
   &~~   
   +\int_{J_j}\int_{I_i}(g(x,t,\bm{u})-g(x,t,\bm{u}_h))\psi dxdydW_t,
     \label{multi2ddgge2}\\
     \int_{J_j}\int_{I_i} d(T-T_h) \phi dxdy
     &= -\int_{J_j}\mathcal{A}_{I_i}(E-E_h,\phi;-\alpha_1)dydt \notag \\
   &~~   
     +\int_{J_j}\int_{I_i}(r(x,t,\bm{u})-r(x,t,\bm{u}_h))\phi dxdydW_t.
     \label{multi2ddgge3}
 \end{align} 
Choose the test functions $\varphi=\xi^E,~~\psi=\xi^S,~~\phi=\xi^T$, and notice that 
\[\int_{J_j}\mathcal{A}_{I_i}(\epsilon^T,\xi^E;\alpha_1)dydt=\int_{I_i}\mathcal{A}_{J_j}(\epsilon^S,\xi^E;-\alpha_2)dxdt=0,\] 
following the definition of the projections. For a function $q\in \{f, g, r\}$ and $w\in\mathbb{V}_h^k$, define 
\[\vect{E}^{i,j}_q(w)=\int_{J_j}\int_{I_i}(q(x,t,\bm{u})-q(x,t,\bm{u}_h))wdxdydW_t, \] 
which is an It\^{o} integral, hence we have $\mathbb{E}\big(\vect{E}^{i,j}_q(w)\big)=0$ for any $w$.
Therefore, equations \eqref{multi2ddgge1}-\eqref{multi2ddgge3} become
\begin{align*}
 &\int_{J_j}\int_{I_i}d\xi^E\xi^Edxdy-\int_{J_j}\int_{I_i}d\epsilon^E\xi^Edxdy
 =-\int_{J_j}\mathcal{A}_{I_i}(\xi^T,\xi^E;\alpha_1)dydt
+\int_{I_i}\mathcal{A}_{J_j}(\xi^S,\xi^E;-\alpha_2) dxdt+\vect{E}^{i,j}_f(\xi^E), \\ %\label{multi2derror1}\\
 & \int_{J_j}\int_{I_i}d\xi^S\xi^Sdxdy-\int_{J_j}\int_{I_i}d\epsilon^S\xi^Sdxdy
  =\int_{I_i}\mathcal{A}_{J_j}(\xi^E,\xi^S;\alpha_2)dxdt
  -\int_{I_i}\mathcal{A}_{J_j}(\epsilon^E,\xi^S;\alpha_2)dxdt+\vect{E}^{i,j}_g(\xi^S), \\ %  \label{multi2derror2}\\
 & \int_{J_j}\int_{I_i}d\xi^T\xi^T dxdy-\int_{J_j}\int_{I_i} d\epsilon^T\xi^T dxdy=
  -\int_{J_j} \mathcal{A}_{I_i}(\xi^E,\xi^T;-\alpha_1)dydt
  +\int_{J_j}\mathcal{A}_{I_i}(\epsilon^E,\xi^T;-\alpha_1)dydt+\vect{E}^{i,j}_r(\xi^T). %\notag \\  \label{multi2derror3}
\end{align*}
Summing up these equations and applying integration by parts, we obtain 
\begin{align}\label{multi2dqqqq}  
    &\int_{J_j}\int_{I_i}d\xi^E\xi^E+d\xi^S\xi^S+d\xi^T\xi^T dxdy~~~~~~~~~~~~~~~~~~~~~~~~~~  \notag \\
    &=    \int_{J_j}\int_{I_i}d\epsilon^E\xi^E+d\epsilon^S\xi^S+d\epsilon^T\xi^T dxdy
    -\int_{J_j}\Bigg(\Pi_{i-\frac12,y}-\Pi_{i+\frac12,y}\Bigg)
    dydt+\int_{I_i}\Bigg(\bar{\Pi}_{x,j-\frac12} -\bar{\Pi}_{x,j+\frac12}\Bigg)dxdt	\notag \\
    &-\int_{I_i}\mathcal{A}_{J_j}(\epsilon^E,\xi^S;\alpha_2)dxdt
  +\int_{J_j}\mathcal{A}_{I_i}(\epsilon^E,\xi^T;-\alpha_1)dydt+\vect{E}^{i,j}_f(\xi^E)+
  \vect{E}^{i,j}_g(\xi^S)+\vect{E}^{i,j}_r(\xi^T),
\end{align}
where 
\[\Pi=\Big(\frac12+\alpha_1\Big)(\xi^T)^+(\xi^E)^-+\Big(\frac12-\alpha_1\Big)(\xi^T)^-(\xi^E)^+,~~ 
\bar{\Pi}=\Big(\frac12+\alpha_2\Big)(\xi^S)^-(\xi^E)^++ \Big(\frac12-\alpha_2\Big)(\xi^S)^+(\xi^E)^-.\]
By It\^{o}'s lemma, we have
\[d(\xi^E)^2=2d\xi^E\xi^E+d\langle \xi^E,\xi^E\rangle_t,~~d(\xi^S)^2=2d\xi^S\xi^S+d\langle \xi^S,\xi^S\rangle_t,~~d(\xi^T)^2=2d\xi^T\xi^T+d\langle \xi^T,\xi^T\rangle_t.\]
Following the same process as in the derivation of \eqref{multidxiu}, we have 
\[\int_{J_j}\int_{I_i}\mathbb{E}\langle \xi^E,\xi^E\rangle _t dxdy\le C\int_{J_j}\int_{I_i}\int_0^t\mathbb{E}(\mathbb{P}^{-\alpha_1,\alpha_2}(f)-f)^2d\tau dxdy+C\mathbb{E}\vect{e}_{i,j},\]
\[\int_{J_j}\int_{I_i}\mathbb{E}\langle \xi^S,\rangle_tdxdy\le C\int_{J_j}\int_{I_i}\int_0^t\mathbb{E}(\mathbb{P}^{-\alpha_2}_y(g)-g)^2d\tau dxdy+C\mathbb{E}\vect{e}_{i,j},\]
\[\int_{J_j}\int_{I_i}\mathbb{E}\langle \xi^T,\xi^T\rangle_t dxdy\le C\int_{J_j}\int_{I_i}\int_0^t\mathbb{E}(\mathbb{P}^{\alpha_1}_x(r)-r)^2d\tau dxdy+C\mathbb{E}\vect{e}_{i,j},\]
where 
\[ \vect{e}_{i,j} := \int_{J_j}\int_{I_i}\int_0^t|\xi^E-\epsilon^E|^2+|\xi^S-\epsilon^S|^2+|\xi^T-\epsilon^T|^2d\tau dxdy. \]
After taking expectation on equation \eqref{multi2dqqqq}, summing over all the cells, and integrating over $t$, we have 
\begin{align*}%\label{multi2dwwwww}
  &\frac12\int_J\int_I \mathbb{E}((\xi^E)^2+ (\xi^S)^2+(\xi^T)^2)dxdy\\
   &=\frac12\int_{J}\int_{I}\mathbb{E}\big(\langle \xi^E,\xi^E\rangle_t+\langle \xi^S,\xi^S\rangle_t+\langle \xi^T,\xi^T\rangle_t\big) dxdy+
   \int_{J}\int_{I}\int_0^t \mathbb{E}(d\xi^E\xi^E+d\xi^S\xi^S+d\xi^T\xi^T) dxdy\\
    &\le\int_0^t\mathbb{E}(\norm{\mathbb{P}^{-\alpha_1,\alpha_2}(f)-f}^2+\norm{\mathbb{P}^{\alpha_1}_x(r)-r}^2+\norm{\mathbb{P}^{-\alpha_2}_y(g)-g}^2)d\tau  \\        
    &
%    +\sum_{i,j}\mathbb{E}\Big(\vect{E}^{i,j}_f(\xi^E)+\vect{E}^{i,j}_g(\xi^S)+\vect{E}^{i,j}_r(\xi^T)\Big)
   +  \int_{J}\int_{I}\int_0^t \mathbb{E}(d\epsilon^E\xi^E+d\epsilon^S\xi^S+d\epsilon^T\xi^T ) dxdy
   + Ch^{k+1}\int_0^t\mathbb{E}\big(\norm{\xi^T} +\norm{\xi^S}\big)d\tau
   + C\sum_{i,j}\mathbb{E}\vect{e}_{i,j}.
\end{align*}
Following the same derivation of Eqs. \eqref{multidepsilones1}-\eqref{multidepsilones2}, we have 
\begin{align*}
   &\int_{J}\int_{I}\int_0^t \mathbb{E}(d\epsilon^E\xi^E+d\epsilon^S\xi^S+d\epsilon^T\xi^T ) dxdy
    \le Ch^{2k+2} \notag \\
   & \hskip2cm
    +\int_0^t \mathbb{E}\Big(\norm{\xi^E(x,y,\tau)}^2+\norm{\xi^S(x,y,\tau)}^2+\norm{\xi^T(x,y,\tau)}^2\Big)d\tau,
\end{align*}
after applying the superconvergence property in Lemma \ref{multisuperconvergence}.
We utilize the projection error property and Young's inequality to obtain 
\[\mathbb{E}\Big(\norm{\xi^E}^2+\norm{\xi^S}^2+\norm{\xi^T}^2\Big)\le C
\int_0^t \mathbb{E}\Big(\norm{\xi^E(x,y,\tau)}^2+\norm{\xi^S(x,y,\tau)}^2+\norm{\xi^T(x,y,\tau)}^2\Big)d\tau+ Ch^{2k+2}.\]
The optimal error estimate \eqref{multi2derroreq} follows from applying the Gronwall's inequality and the optimal projection error.
\end{proof}

\section{Temporal discretization}\label{multitime}
\setcounter{equation}{0}\setcounter{figure}{0}\setcounter{table}{0}

After DG spatial discretization, the semi-discrete DG methods \eqref{multiphi}-\eqref{multiphii}, or \eqref{multitridg1}-\eqref{multitridg2} can be obtained.
To solve the resulting stochastic differential equations, we present Taylor 2.0 strong scheme \cite{YL 2020} as the temporal discretization in this section. 

Let us consider the general matrix-valued stochastic differential equations of the form
\begin{eqnarray}\label{multiSDE}
  \begin{cases}
    d X_t^{i,j} = a^{i,j}(X_t)dt+b^{i,j}(X_t)dW_t, ~~t>0\\
    X_0^{i,j} = x_0^{i,j},
  \end{cases}
\end{eqnarray}
with $i=0,1,\cdots,m$ and $j=0,1,\cdots, M+1$. Suppose that $Y^{i,j}_n$ is a numerical approximation of $X^{i,j}_{t_n}$ at the time $t_n$, and $Y^{i,j}_0=x^{i,j}_0$.  
The recurrent equation between $Y_n^{i,j}$ and $Y_{n+1}^{i,j}$ is derived below. Define the following operators: 
 \[\mathcal{L}^0f=\sum^{M+1}_{j=0}\sum^m_{i=0}a^{i,j}\frac{\partial f}{\partial x_{i,j}}+\frac12 \sum^{M+1}_{l,j=0}\sum^{m}_{m,i=0}b^{i,j}b^{m,l}\frac{\partial^2 f}{\partial x_{i,j}\partial x_{m,l}},\]
 and 
 \[\mathcal{L}^1f=\sum_{j=0}^M\sum_{i=0}^m b^{i,j}\frac{\partial f}{\partial x_{i,j}},\]
where $f: \mathbb{R}^{(m+1)\times(M+2)}\to \mathbb{R}$ is a twice differentiable function. 
As studied in \cite[Section 10.5]{PK 1999} and \cite[Appendix A]{YL 2020}, the following Taylor scheme has 2.0 order of convergence:
 \begin{eqnarray}
   \begin{split}
     Y_{n+1}^{i,j}&=Y_n^{i,j}+a^{i,j}(Y_n)\tau+b^{i,j}(Y_n)\Delta W+\frac12 \mathcal{L}^1 
     b^{i,j}(Y_n)\Big((\Delta W)^2-\tau\Big)\\
     &+\frac12\mathcal{L}^0 a^{i,j}(Y_n)\tau^2+\mathcal{L}^0 b^{i,j}(Y_n)(\Delta W\tau-\Delta Z)
     +\mathcal{L}^1 a^{i,j}(Y_n)\Delta Z\\
     &+\frac16 \mathcal{L}^1\mathcal{L}^1 b^{i,j}(Y_n)\Big((\Delta W)^2-3\tau\Big)\Delta 
     W+\mathcal{L}^1\mathcal{L}^0 b^{i,j}(Y_n)(-\Delta U+\Delta W\Delta Z)\\
     &+\mathcal{L}^1\mathcal{L}^1 a^{i,j}(Y_n)\Big(\frac12 \Delta U-\frac14 \tau^2\Big)
     +\mathcal{L}^0\mathcal{L}^1b^{i,j}(Y_n)\Big(\frac12 \Delta U-\Delta W\Delta Z+\frac12 (\Delta W)^2\tau-\frac14 
     \tau^2\Big)\\
     &+\frac{1}{24}\mathcal{L}^1\mathcal{L}^1\mathcal{L}^1 b^{i,j}(Y_n)\Big((\Delta W)^4-6(\Delta 
     W)^2\tau+3\tau^2\Big),
   \end{split}
 \end{eqnarray}
 where $$\tau=t_{n+1}-t_n, \quad \Delta W=W_{t_{n+1}}-W_{t_n},$$
 and 
 $$\Delta Z=\int_{t_n}^{t_{n+1}}W_s-W_{t_{n}}ds, \quad \Delta U=\int_{t_n}^{t_{n+1}}(W_s-W_{t_n})^2ds.$$
% \YX{How did you compute the derivatives in the $\mathcal{L}^1$ and $\mathcal{L}^0$ terms? [12] used the derivative-free approach.}
 
To numerically compute the stochastic variables $\Delta W$, $\Delta Z$ and $\Delta U$, we define a new process 
 \[v(s)=\frac{W_{t_n+\tau s}-W_{t_n}}{\sqrt{\tau}}, \qquad s\in[0,1], \]
 and have 
 \[\Delta W=\tau^{\frac12}v(1),  ~\Delta Z=\tau^{\frac32}\int_0^1v(s)ds, ~\Delta U=\tau^2\int_0^1v^2(s)ds, \]
 which can be evaluated by solving the following system of equations 
 \begin{eqnarray}\label{multinoiseode}
   \begin{cases}
     dx=dv(s), \qquad x(0)=0,\\
     dy=xds, \qquad  y(0)=0,\\
     dz=x^2ds, \qquad  z(0)=0.
   \end{cases}
 \end{eqnarray}
 at the moment $s=1$.
The system \eqref{multinoiseode} can be solved numerically to obtain the approximation of $\Delta W$, $\Delta Z$ and $\Delta U$. 
We refer to \cite[Appendix A.2.2]{YL 2020} for the details of implementation.  
% \YX{Still need to add some details. It is not clear how $v(s)$, $\int v(s)$ is computed numerically.}
 
\section{Numerical Experiment}\label{multinum}
\setcounter{equation}{0}\setcounter{figure}{0}\setcounter{table}{0}

In this section we present numerical results of the proposed scheme for the one-dimensional and two-dimensional stochastic Maxwell equations with multiplicative noise. We use polynomials of degree $k$ in our proposed DG methods for spatial discretization, and Taylor 2.0 strong scheme for the temporal discretization. The accuracy tests are provided for both 1D case and 2D case to demonstrate the convergence rate. The time history of the discrete energy is also presented for both examples. The Monte-Carlo method is used to compute the stochastic term, and the expectation is computed by averaging over all the samples.

We first consider the following one-dimensional equations
  \begin{equation*}
  \begin{cases}
dv=-u_xdt+v dW_t, &\\
du=-v_xdt+udW_t,
\end{cases}
\end{equation*}
with periodic boundary conditions and the exact solutions given by
\begin{eqnarray}\label{multinumex1}
  \begin{cases}v=(\sin(x-t)+\cos(x+t))e^{W_t-\frac12 t},\\
u=(\sin(x-t)-\cos(x+t))e^{W_t-\frac12 t}.\end{cases}
\end{eqnarray}
The computational domain is set to be $[0,2\pi]$, and the final time $T=0.5$. 
Nx and Nt are used to denote the number of space steps and time steps, and we use uniform meshes as spatial discretization. 
The numerical initial condition is taken as the projection of the exact solutions \eqref{multinumex1} at $t=0$. We apply Monte Carlo 
simulation with 3000 samples to approximate the expectation. Table \ref{multit11} 
and Table \ref{multit12} show the convergence rate when $k=1$ and $k=2$ respectively. 
In both examples $\Delta t$ is chosen small enough to ensure the spatial error dominates, 
and we can observe the optimal error estimate (the expected $(k+1)$-th order of convergence), 
which is consistent with the result in Theorem \ref{multierror} for the semi-discrete DG method. 

\begin{table}[H]
     \caption{Numerical error and convergence rates of 1D case when $k=1$.}
\centering
\begin{tabular}{|c|c|c|c|c|c|}
\hline 
Nx&Nt&$\big(\mathbb{E}\norm{e_u}^2\big)^{1/2}$&rate&$\big(\mathbb{E}\norm{e_v}^2\big)^{1/2}$&rate\\
\hline 
20&200&0.02537&0&7.855E-3&0\\
\hline
40&400&6.407E-3&1.9851&1.978E-3&1.9893\\
\hline
80&800&1.567E-3&2.0313&4.837E-4&2.0320\\
\hline
160&1600&3.913E-4&2.0022&1.206E-4&2.0046\\
\hline

\end{tabular}\label{multit11}
\end{table}
\begin{table}[H]
     \caption{Numerical error and convergence rates of 1D case when $k=2$.}
\centering
\begin{tabular}{|c|c|c|c|c|c|}
\hline 
Nx&Nt&$\big(\mathbb{E}\norm{e_u}^2\big)^{1/2}$&rate&$\big(\mathbb{E}\norm{e_v}^2\big)^{1/2}$&rate\\
\hline 
20&200&6.254E-4&0&2.073E-4&0\\
\hline
40&400&7.970E-5&2.9722&2.383E-5&3.1204\\
\hline
80&800&9.808E-6&3.0225&3.127E-6&2.9302\\
\hline
160&1600&1.269E-6&2.9500&4.071E-7&2.9412\\
\hline

\end{tabular}\label{multit12}
\end{table}

Next we consider the two-dimensional stochastic Maxwell equations
 \begin{equation*}
  \begin{cases}
dE-T_xdt+S_ydt=EdW_t,\\
dS+E_ydt=SdW_t,\\
dT-E_xdt=TdW_t,
\end{cases}
\end{equation*}
with periodic boundary conditions. The exact solutions take the form
\begin{eqnarray}
  \begin{cases}\label{multi2dnumex}
E=\big(\sin(x+t)-\cos(y+t)\big)e^{W_t-\frac12 t},\\
S=\sin(x+t)e^{W_t-\frac12 t},\\
T=\cos(y+t)e^{W_t-\frac12 t}. \end{cases}
\end{eqnarray} 
The space domain is $[0,2\pi]^2$. The numerical initial condition is taken as the projection of the exact solutions \eqref{multi2dnumex} at $t=0$. 
Nx, Ny and Nt are used to denote the number of space cells in $x$, $y$ directions, and the number of time steps, 
and uniform rectangular meshes are considered as spatial discretization. 
We run the simulation until final time $T=0.1$. Monte Carlo simulation with 500 samples are used to approximate the expectation. 
Table \ref{multi2dt1} and Table \ref{multi2dt2} present the convergence rate for the case of $k=1$ and $k=2$ respectively. 
We can observe that in both cases, the optimal convergence rate is achieved.

\begin{table}[H]
     \caption{Numerical error and convergence rates of 2D case when $k=1$. }
\centering
\begin{tabular}{|c|c|c|c|c|c|c|c|c|}
\hline 
Nx&Ny&Nt&$(\mathbb{E}\norm{e_E}^2)^{1/2}$&rate&$(\mathbb{E}\norm{e_S}^2)^{1/2}$&rate&$(\mathbb{E}\norm{e_T}^2)^{1/2}$&rate\\
\hline 
20&20&20&0.02582&0&0.01807&0&0.01807&0\\
\hline
40&40&40&6.690E-3&1.9484&4.674E-3&1.9510&4.674E-3&1.9510\\
\hline
80&80&80&1.648E-3&2.0215&1.160E-3&2.0106&1.160E-3&2.0106\\
\hline
160&160&160&3.984E-4&2.0482&2.810E-4&2.0453&2.810E-4&2.0453\\
\hline
\end{tabular}\label{multi2dt1}
\end{table}

\begin{table}[H]
     \caption{Numerical error and convergence rates of 2D case when $k=2$.}
\centering
\begin{tabular}{|c|c|c|c|c|c|c|c|c|}
\hline 
Nx&Ny&Nt&$(\mathbb{E}\norm{e_E}^2)^{1/2}$&rate&$(\mathbb{E}\norm{e_S}^2)^{1/2}$&rate&$(\mathbb{E}\norm{e_T}^2)^{1/2}$&rate\\
\hline 
20&20&20&8.029E-4&0&5.741E-4&0&5.741E-4&0\\
\hline
40&40&40&1.052E-4&2.9317&7.509E-5&2.9347&7.509E-5&2.9347\\
\hline
80&80&80&1.254E-5&3.0695&9.113E-6&3.0425&9.113E-6&3.0425\\
\hline
160&160&160&1.572E-6&2.9952&1.179E-6&2.9507&1.179E-6&2.9507\\
\hline
\end{tabular}\label{multi2dt2}
\end{table}

The discrete energy law satisfied by the numerical solutions was studied in Theorem \ref{multithm1} for one-dimensional system, and in Theorems \ref{multitrieng} and \ref{multi2denergy} for two-dimensional system. In Figure \ref{multienergy}, the time history of averaged energy is shown for two cases. 

\begin{figure}[h]
\centering
\subfigure{\includegraphics[width=0.35\textwidth]{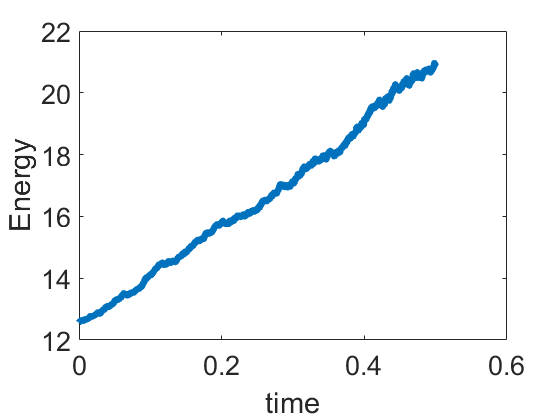}}
\subfigure{\includegraphics[width=0.35\textwidth]{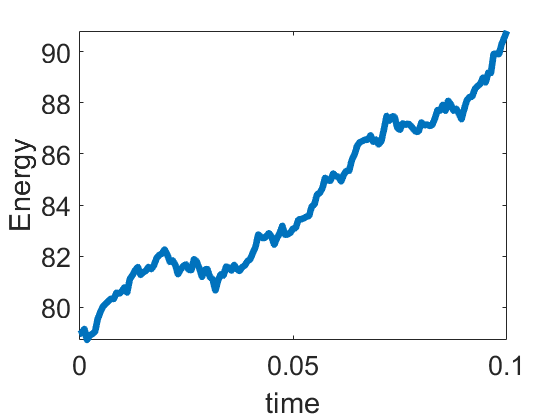}}
\caption{The time history of the averaged energy. Left: 1D result; Right: 2D result.}
\label{multienergy}
\end{figure}

\section{Conclusion Remarks} \label{multisec6}
\setcounter{equation}{0}\setcounter{figure}{0}\setcounter{table}{0}

In this paper we applied high order DG scheme for one- and two-dimensional stochastic Maxwell equations with multiplicative noise. We provide the semi-discrete energy law for both cases. Optimal error estimate of the semi-discrete method is obtained for one-dimensional case, and two-dimensional case on both rectangular meshes and triangular meshes under certain mesh assumptions. The semi-discrete method is combined with strong Taylor 2.0 temporal discretization, and numerical results are presented to validate the optimal error estimates and the growth of energy.

\end{document}